\documentclass[preprint]{imsart}
\RequirePackage{amsthm,amsmath,amsfonts,amssymb,mathtools,mathtools}
\RequirePackage[authoryear]{natbib} %% uncomment this for author-year bibliography
\RequirePackage[colorlinks,citecolor=blue,urlcolor=blue]{hyperref}

\startlocaldefs
\numberwithin{equation}{section}
\newtheorem{thm}{Theorem}[section]
\newtheorem{lemma}{Lemma}[section]
\newtheorem{corollary}{Corollary}[section]
\theoremstyle{remark}
\newtheorem{remark}{Remark}[section]
\newtheorem*{mycomment}{Comment}
\def\citeapos#1{\citeauthor{#1}'s (\citeyear{#1})}
\newcommand{\mymid}{\!\mid\!}
\newcommand{\mbR}{\mathbb{R}}
\newcommand{\rd}{\mathrm{d}}
\newcommand{\dps}{\displaystyle}

\endlocaldefs

\begin{document}

\begin{frontmatter}
 \title{Admissible estimators of a multivariate normal mean vector when the scale is unknown}
%\title{A sample article title with some additional note\thanksref{T1}}
\runtitle{Admissible estimators}
%\thankstext{T1}{A sample of additional note to the title.}
\runauthor{Y. Maruyama and W. Strawderman}

\begin{aug}
%%%%%%%%%%%%%%%%%%%%%%%%%%%%%%%%%%%%%%%%%%%%%%
%%Only one address is permitted per author. %%
%%Only division, organization and e-mail is %%
%%included in the address.                  %%
%%Additional information can be included in %%
%%the Acknowledgments section if necessary. %%
%%%%%%%%%%%%%%%%%%%%%%%%%%%%%%%%%%%%%%%%%%%%%%
\author[A]{\fnms{Yuzo} \snm{Maruyama}\ead[label=e1]{maruyama@mi.u-tokyo.ac.jp}}
\and
\author[B]{\fnms{William, E.} \snm{Strawderman}\ead[label=e2]{straw@stat.rutgers.edu}}
%%%%%%%%%%%%%%%%%%%%%%%%%%%%%%%%%%%%%%%%%%%%%%
%% Addresses                                %%
%%%%%%%%%%%%%%%%%%%%%%%%%%%%%%%%%%%%%%%%%%%%%%
\address[A]{University of Tokyo, \printead{e1}}
\address[B]{Rutgers University, \printead{e2}}
\end{aug}

\begin{abstract}
 We study admissibility of a subclass of generalized Bayes estimators of a multivariate normal vector
 when the variance is unknown, under scaled quadratic loss.
Minimaxity is also established for certain of these estimators.
\end{abstract}

\begin{keyword}[class=MSC2010]
\kwd[Primary ]{62C15}  
%\kwd{???}
\kwd[; secondary ]{62C20}
\end{keyword}

\begin{keyword}
\kwd{admissibility}
\kwd{Bayes estimators}
\kwd{minimaxity}
\end{keyword}

\end{frontmatter}

%%%% Main text entry area:
 \section{Introduction}
\label{sec:intro}
We study admissibility of a subclass of generalized Bayes estimators of a multivariate normal vector in the case of an unknown variance, under scaled squared error loss.
Specifically, let $X$ and $S$ be independent with
\begin{equation}\label{our.model}
 X \sim N_p(\theta,\sigma^2 I_p), \ S\sim \sigma^2\chi_n^2,
\end{equation}
and consider estimation of $\theta$ under the scaled quadratic loss
\begin{equation}\label{loss}
 L(d;\{\theta,\sigma^2\})=\frac{\|d-\theta\|^2}{\sigma^2}.
\end{equation}
The class of hierarchical priors we consider, with $\eta=1/\sigma^2$, is $\pi(\theta,\eta)$ given by
\begin{equation}\label{our.prior}
\begin{split}
\theta \mymid \{\eta,\lambda \}& \sim N_p\left(0,\frac{1}{\eta}\frac{1-\lambda}{\lambda}I_p\right),\quad 
\lambda \sim \frac{\lambda^a(1-\lambda)^b}{B(a+1,b+1)} \text{ for }0<\lambda<1 \\
 \eta & \sim 1/\eta \text{ for }0<\eta<\infty,
\end{split}
\end{equation}
where $a>-1$ and $b>-1$.
Hence the prior is improper, due to the impropriety of the invariant prior on $\eta$.
However, the conditional prior on $\theta\mymid \eta$ is a proper scale mixture of normal priors generalizing that in \cite{Strawderman-1973} (see also \cite{DSW-2018}), who considered proper priors on $\eta$.
\cite{Maruyama-Strawderman-2020} considered related priors in a study of admissibility within the class of scale equivariant estimators.
\cite{Maru-Straw-2005} studied minimaxity of such estimators.
This paper may be viewed as a companion paper to \cite{Maruyama-Strawderman-2020} where admissibility is proved among the class of all estimators, not just scale equivariant ones.
We also address the issue of minimaxity.

To our knowledge, aside from proper Bayes estimators, these are the first results proving admissibility of generalized Bayes estimators in this problem for $p>2$.
A difficulty comes from the presence of the nuisance parameter, $\eta$.
As mentioned in \cite{James-Stein-1961} and \cite{Brewster-Zidek-1974}, proving admissibility of generalized Bayes estimators in the presence of a nuisance parameter has been a longstanding unsolved problem. 

Our method of proof is similar in spirit to that of \cite{Brown-Hwang-1982} in the known scale case and makes use of a version of \citeapos{Blyth-1951} method.
As \cite{Berger-1985} pointed out, ``Indeed, in general, very elaborate (and difficult to work with) choices of sequences of proper priors are needed'' for proving admissibility.
The sequence of proper priors, we use in this paper, is given by
 \begin{align*}
\pi(\theta,\eta)  \left\{\frac{i}{i+\log (\max(\eta,1/\eta))}\right\}^2,\ i=1,2,\dots,
 \end{align*}
 which is novel in this area, to our knowledge.
 The main result of this paper (Theorem \ref{thm.main}) is that the generalized Bayes estimator corresponding to the hierarchical prior \eqref{our.prior} is admissible provided $-1<a<n/2$ and $b>-1/2$.
 We also show that, for $p\geq 5$, a subclass of these estimators are admissible and minimax.
 Minimaxity of some of these estimators was shown in \cite{Maru-Straw-2005}. 
 Among them, the most striking estimator, because of its simple and explicit form, is
 \begin{equation*}
 \left(1-\frac{2(p-2)/(n+2)}{\|X\|^2/S+1+2(p-2)/(n+2)}\right)X.
 \end{equation*}
We show this to be admissible and minimax when $n>3$ and $p>4n/(n-2)$.
  
\cite{Brown-1971} largely settled the issue of admissibility in the known scale case and has been a motivating force behind many admissibility studies in multidimensional settings, including \cite{Brown-Hwang-1982} and of course, the present paper.
  \cite{Johnstone-2019} gives an excellent review of the development and impact of Brown's monumental paper.
  
  The main result on admissibility is given in Section \ref{sec:main.adm}.
  Minimaxity is discussed in Section \ref{sec:minimax}.
  An appendix is devoted to the proofs of many of the results used in the development of Section \ref{sec:main.adm}.
  Comments are given in Section \ref{sec:cr}.

 \section{Admissibility of generalized {B}ayes estimators}
\label{sec:main.adm}
The main result of this paper is the following.
\begin{thm}\label{thm.main}
The generalized Bayes estimator corresponding to the prior $ \pi(\theta,\eta)$ given by \eqref{our.prior} is admissible for the model \eqref{our.model} under the loss \eqref{loss} provided
\begin{equation}\label{main.cond}
 -1<a<n/2 \text{ and }  b>-1/2.
\end{equation}
\end{thm}
The basic structure of the proof is standard, as in \cite{Brown-Hwang-1982}, and is based on the \cite{Blyth-1951} method.
We give an increasing sequence of proper priors
\begin{equation*}
 \pi_i(\theta,\eta)=  \pi(\theta,\eta)h_i^2(\eta).
\end{equation*}
Let $ \delta_\pi $ be the generalized Bayes estimator corresponding to $ \pi(\theta,\eta)$ and $ \delta_{\pi i}$ the proper Bayes estimator corresponding to $ \pi_i(\theta,\eta)$.
The Bayes risk difference of $ \delta_\pi $ and $ \delta_{\pi i}$ with respect to the prior $ \pi_i(\theta,\eta)$, $\Delta_i$, is defined by
\begin{align}\label{bill.0}
\Delta_i =
% \iint
\int_{\mbR^p}\int_0^\infty 
 \left\{E\left[\eta\|\delta_{\pi }-\theta\|^2\right]-E\left[\eta\|\delta_{\pi i}-\theta\|^2\right] \right\}
\pi_i(\theta,\eta)
 \rd  \theta \rd \eta.
\end{align}
Then we show that
\begin{equation}\label{bill.1}
 \lim_{i\to\infty}\Delta_i =0.
\end{equation}
%and $R(\theta,\eta,\delta_{\pi })$ (resp.~$R(\theta,\eta,\delta_{\pi i})$) is the risk of $\delta_{\pi }$ (resp.~$\delta_{\pi i}$).
The following form of Blyth's sufficient condition shows that \eqref{bill.1} implies admissibility. The proof, as for the lemmas that follow, is given in Appendix.
\begin{lemma}\label{lem.Blyth}
 A sufficient condition for $\delta_\pi$ to be admissible is that there exists an increasing (in $i$) sequence of proper priors $\pi_i(\theta,\eta)$ such that $\pi_i(\theta,\eta)>0$ for all $\theta,\eta$ and that $\Delta_i$ satisfies \eqref{bill.1}.
\end{lemma} 
The following two lemmas give the form of the Bayes estimator and an expression for $\Delta_i$.
\begin{lemma}\label{lem.form.Bayes}
 The Bayes estimator of $\theta$ for the problem in \eqref{our.model}, \eqref{loss} and \eqref{our.prior} is given by
 \begin{align*}
 \delta_\pi (x ,s)  =x  + \frac{m(\nabla_{\theta} \pi)}{m(\pi\eta)},
 \end{align*}
where $\nabla_{\theta}=(\partial/\partial \theta_1,\partial/\partial \theta_2,\dots ,\partial/\partial \theta_p )'$ and
\begin{align*}
 m(\psi)  &= \int_{\mbR^p}\int_0^\infty 
\psi(\theta,\eta)f_{x }(x \mymid\theta,\eta)f_s(s\mymid\eta)
   \rd \theta \rd \eta, \\
 f_{x}(x\mymid\theta,\eta)&=\frac{\eta^{p/2}}{(2\pi)^{p/2}}\exp\left(-\frac{\eta}{2}\|x -\theta\|^2\right), \\
f_s(s\mymid\eta)&=\frac{\eta^{n/2}}{\Gamma(n/2)2^{n/2}}s^{n/2-1}\exp\left(-\frac{\eta s}{2}\right).
\end{align*}
\end{lemma}
\begin{lemma}\label{lem.Bayes.risk.diff}
 The Bayes risk difference $\Delta_i$ is written as
\begin{align}\label{Delta_i.0}
\Delta_i=\int_{\mbR^{p}}\int_0^\infty \left\| \frac{m(\nabla \pi )}{m(\pi \eta)} - 
\frac{m(h_i^2\nabla \pi )}{m(\pi h_i^2\eta)}\right\|^2 m(\pi h_i^2\eta)\rd x  \rd s .
\end{align} 
\end{lemma}
\begin{remark}
 Recall $h_i$ does not depend on $\theta$. 
Hence, unlike \cite{Brown-Hwang-1982}, there is no term
\begin{align*}
 \frac{m(\{\nabla_\theta h_i^2\} \pi )}{m(\pi h_i^2\eta)}
\end{align*}
in the expression of \eqref{Delta_i.0}. 
\end{remark}
The specific sequence of priors used to prove admissibility is given by
\begin{equation}\label{our.prior.i}
 \pi_i(\theta,\eta)=  \pi(\theta,\eta)h_i^2(\eta) 
\end{equation}
where $\pi(\theta,\eta)$ is given in \eqref{our.prior}, and
\begin{equation}\label{eq:original_h_i}
 h_i(\eta)=\frac{i}{i+\log (\max(\eta,1/\eta))}=
\begin{cases}
\dps \frac{i}{i+\log (1/\eta)} & 0<\eta<1, \\
\dps \frac{i}{i+\log \eta} & \eta\geq 1.
\end{cases}
\end{equation}
 Note that $h_i(\eta)$ is increasing in $i$ and 
\begin{align*}
 \lim_{i\to\infty}h_i(\eta)=1 \text{ for all }\eta>0.
\end{align*}
\begin{lemma}\label{lem.hi}
\begin{enumerate}
 \item $ \dps\int_0^\infty \eta^{-1}h_i^2(\eta)\rd \eta=2i$.
\item  The prior $\pi_i(\theta,\eta)$ is proper for all $i$.
\end{enumerate}
\end{lemma}

The bulk of the remainder of the construction of the proof consists in showing that (a) the integrand of $\Delta_i$ in \eqref{Delta_i.0} is bounded by an integrable function and (b) the integrand itself tends to $0$ as $i$ tends to infinity.
Hence, by the dominated convergence theorem, \eqref{bill.1} is satisfied so that $\delta_\pi$ is admissible.

\begin{lemma}\label{lem.integrand.converge}
 The integrand of $\Delta_i$ in \eqref{Delta_i.0} converges to $0$ as $i$ tends to infinity.
\end{lemma}
Most of the remaining lemmas are devoted to bounding the integrand of \eqref{Delta_i.0}.
\begin{lemma}\label{lem.bound.delta}
 There exists a positive constant $C$ such that
\begin{align*}
\left\| \frac{m(\nabla \pi )}{m(\pi \eta)} - 
 \frac{m(h_i^2\nabla \pi )}{m(\pi h_i^2\eta)}\right\|^2 m(\pi h_i^2\eta)  
\leq C A(x,s) B_i (x,s)
\end{align*}
 where
\begin{equation}\label{ABI}
  A(x,s)=m\left(\frac{\pi(\theta,\eta) }{\eta\|\theta\|^2}k(\eta\|\theta\|^2)\right), \ 
  B_i (x,s)=1- \frac{m(\pi \eta h_i)^2}{m(\pi \eta)m(\pi h_i^2\eta)}
\end{equation}
and $k(r)=r^{1/2}I_{[0,1]}(r)+I_{(1,\infty)}(r)$.
\end{lemma}
The next lemma gives a bound on $A(x,s)$.
\begin{lemma}\label{lem.bound.A}
If $b>-1/2$,
 \begin{align*}
 A(x,s)%=m\left(\frac{\pi(\theta,\eta) }{\eta\|\theta\|^2}k(\eta\|\theta\|^2)\right) 
\leq Ds^{-p/2-1}\left(1+\|x\|^2/s\right)^{-p/2-1/2}
 \end{align*}
for some positive constant $D$.
\end{lemma}
The following gives a bound on $B_i (x,s)$ and accounts for the condition on $a$ in Theorem \ref{thm.main} as will be seen in the proof of Theorem \ref{thm.main}.
\begin{lemma}\label{lem.bound.B}
Assume $a<n/2$. For all positive integers $i$, 
 \begin{align*}
  B_i (x,s)\leq 
\begin{cases}
 E/\{1+\log (1/s) \}^2 & s<\gamma_1 \\
 1 & \gamma_1\leq s \leq \gamma_2 \\
 F/(1+\log s )^2 & s>\gamma_2,
\end{cases}
 \end{align*}
for some positive constants $E,F,\gamma_1<1,\gamma_2>1$, all of which are independent of $i$. 
\end{lemma}
\subsection{Proof of Theorem \ref{thm.main}}
Since the integrand of $\Delta_i$ tends to $0$ as $i$ tends to infinity, it follows that $\Delta_i\to 0$ provided for all $i$ that the integrand is bounded by an integrable function.
By Lemmas \ref{lem.Bayes.risk.diff}, \ref{lem.bound.delta}, \ref{lem.bound.A} and \ref{lem.bound.B},
\begin{align*}
 \Delta_i&=\int_{\mbR^{p}}\int_0^\infty \left\| \frac{m(\nabla \pi )}{m(\pi \eta)} - 
\frac{m(h_i^2\nabla \pi )}{m(\pi h_i^2\eta)}
 \right\|^2 m(\pi h_i^2\eta)\rd x  \rd s \\
 &\leq C\int_{\mbR^{p}}\int_0^\infty A(x,s)B_i (x,s) \rd x  \rd s \\
 &\leq CD \int_{\mbR^{p}}\int_0^\infty\left\{s^{-p/2}
\left(1+\|x\|^2/s\right)^{-p/2-1/2}
 \right\}\\ &\qquad \times \left\{ \frac{EI_{(0,\gamma_1)}(s)}{s\{1+\log (1/s) \}^2}+
\frac{I_{[\gamma_1,\gamma_2]}(s)}{s}+\frac{FI_{(\gamma_2,\infty)}(s)}{s\{1+\log s \}^2}  \right\} \rd x  \rd s\\
&=CDGH
\end{align*}
where
\begin{align*}
 G&=
 \int_{\mbR^{p}}s^{-p/2}
 \left(1+\|x\|^2/s\right)^{-p/2-1/2}
 \rd x=\int_{\mbR^p}  \left(1+\|y\|^2\right)^{-p/2-1/2}\rd y , \\
\text{and }\ H&=\int_0^\infty\left\{ \frac{EI_{(0,\gamma_1)}(s)}{s\{1+\log (1/s) \}^2}+
\frac{I_{[\gamma_1,\gamma_2]}(s)}{s}+\frac{FI_{(\gamma_2,\infty)}(s)}{s\{1+\log s \}^2}  \right\} \rd s.
\end{align*}
Note
\begin{align*}
G= \frac{\pi^{p/2}}{\Gamma(p/2)} \int_0^\infty u^{p/2-1}(1+u)^{-p/2-1/2}\rd u =\frac{\pi^{p/2}}{\Gamma(p/2)}B(p/2,1/2)<\infty
\end{align*}
and
\begin{align*}
H&=E\int_0^{\gamma_1}\frac{\rd s}{s\{1+\log (1/s) \}^2}+
 \int_{\gamma_1}^{\gamma_2}\frac{\rd s}{s} +
 F\int_{\gamma_2}^\infty\frac{\rd s}{s(1+\log s)^2} \\
 &=\frac{E}{1+\log (1/\gamma_1)} +\log\frac{\gamma_2}{\gamma_1}+\frac{F}{1+\log \gamma_2}<\infty .
\end{align*}
Hence, by the dominated convergence theorem, we conclude $\Delta_i\to 0$.
This completes the proof of Theorem \ref{thm.main}.
 \section{Minimaxity of generalized {B}ayes estimators}
\label{sec:minimax}
\cite{Maru-Straw-2005} and \cite{Maruyama-Strawderman-2009} discussed minimaxity of the generalized Bayes estimators corresponding to $\pi(\theta,\eta)$ given by \eqref{our.prior}.
Recall that \citeapos{Baranchik-1970} sufficient condition for a shrinkage estimator
\begin{align*}
\left(1-\frac{\phi(\|x\|^2/s)}{\|x\|^2/s}\right)x
\end{align*}	
to be minimax is that a) $\phi(w)$ is non-decreasing, and b) $ 0\leq \phi(w)\leq 2(p-2)/(n+2)$.
The following lemma (Lemma 2.2 of \cite{Maru-Straw-2005}) summarizes the behavior of the generalized Bayes estimators corresponding to $\pi(\theta,\eta)$ given by \eqref{our.prior}.
\begin{lemma}
 \begin{enumerate}
  \item The generalized Bayes estimator corresponding to the prior \eqref{our.prior} is of the form
\begin{align*}
\delta_\pi(x,s)=\left(1-\frac{\phi_\pi(\|x\|^2/s)}{\|x\|^2/s}\right)x.
\end{align*}	
  \item Assume $-p/2-1<a<n/2-1$ and $b\geq 0$.
	Then $ \phi_\pi(w)$ is increasing and approaches $(p/2+a+1)/(n/2-a-1)$ as $w\to\infty$.
 \end{enumerate}
\end{lemma}
Noting
\begin{equation}\label{a.minimax}
 0< \frac{p/2+a+1}{n/2-1-a}\leq 2\frac{p-2}{n+2}\ \Leftrightarrow \ -\frac{p}{2}-1<a\leq \xi(p,n)
  \ \left(<\frac{n}{2}-1\right)
\end{equation}  
	where
\begin{equation}
\xi(p,n)=-2+\frac{(p-2)(n+2)}{2(2p+n-2)},
\end{equation}
we have the minimaxity result which is essentially given in Theorem 2.3 of \cite{Maru-Straw-2005}.
\begin{lemma}
$\delta_\pi(x,s)$ is minimax provided $b\geq 0$ and $ -p/2-1< a\leq \xi(p,n)$.
\end{lemma}
Recall that we assumed $a>-1$ for establishing admissibility in Theorem \ref{thm.main}.
Since $n\geq 3$ and $p>4n/(n-2)$ implies $\xi(p,n)>-1$, we have the following result. 
\begin{thm}\label{thm.adm.mini}
 Suppose $n\geq 3$ and $p>4n/(n-2)$.
 Then the generalized Bayes estimator is minimax and admissible provided
 \begin{equation*}
  -1<a\leq \xi(p,n)\text{ and }b\geq 0.
 \end{equation*}
\end{thm}
A particularly interesting case is $b=n/2-a-2$.
As pointed out in (2.7) of \cite{Maru-Straw-2005}, the generalized Bayes estimator corresponding to the prior \eqref{our.prior} with $b=n/2-a-2$ has the simple closed form, as a variant of the James--Stein estimator, given by
\begin{equation*}
 \left(1-\frac{(p/2+a+1)/(n/2-1-a)}{\|X\|^2/S+1+\{(p/2+a+1)/(n/2-1-a)\}}\right)X.
\end{equation*}
As a corollary of Theorem \ref{thm.adm.mini}, we have the following result.
\begin{corollary}\label{cor.simple}
 Suppose $n\geq 3$ and $p>4n/(n-2)$.
 Then
 \begin{equation*}
 \left(1-\frac{2(p-2)/(n+2)}{\|X\|^2/S+1+2(p-2)/(n+2)}\right)X.
 \end{equation*}
is admissible and minimax.
\end{corollary}

\section{Concluding remarks}
\label{sec:cr}
We have studied admissibility of generalized Bayes estimators of a multivariate normal mean vector in the presence of an unknown scale under scaled squared error loss.
The hierarchical prior structure is proper on $\theta$ given $\eta$ ($=1/\sigma^2$) and is the improper invariant prior ($\pi(\eta)=1/\eta$) on the scale parameter.
We have, to our knowledge, given the first class of improper generalized Bayes admissible estimators for this problem for $p>2$. We note, for $p=1,2$, $X$ is known to be admissible, since it is admissible for each fixed $\sigma^2$ (See for example Lemma 5.2.12 of \cite{Lehmann-Casella-1998}).
The results in this paper are complementary to those in our earlier paper,  \cite{Maruyama-Strawderman-2020}, which gives a class of generalized Bayes estimators that are admissible within the class of scale equivariant estimators.
This paper thereby makes substantial progress in an important problem that has long resisted progress.
Generally, finding admissible procedure in problems with nuisance parameters has been difficult, and the current problem is no exception.
To our knowledge, the only known admissible procedure in the current problem have been proper Bayes estimators, with the exception mentioned above.

Earlier we (\cite{Maru-Straw-2005}) also studied minimaxity of generalized Bayes minimax estimators.
Happily, the intersection of the classes of procedures, studied in these papers is not empty as shown in Section \ref{sec:minimax}.
In particular the estimator given in Corollary \ref{cor.simple} is one such admissible minimax estimator which has a surprisingly simple and explicit form.

%HERE ARE SOME POINTS FOR Bill!
%\begin{enumerate}
% \item In this paper, we assumed 
%       the prior on $\theta\mymid \eta$ is a proper beta mixture of normal priors.
% \item We can easily make a conjecture that any proper prior on $\theta\mymid \eta$ leads admissibility result. 
%\item Again, 	
%this paper may be viewed as a companion paper to \cite{Maruyama-Strawderman-2020} where
%      admissibility is proved among the class of all estimators, not just scale equivariant ones.
%      Roughly speaking, \cite{Maruyama-Strawderman-2020} showed that
%      generalized Bayes estimator corresponding to prior with improper $\theta\mymid \eta$ satisfying \citeapos{Brown-1971} sufficient condition for admissibility in known scale case
%      is admissible among the class of scale equivariant estimators.
%We hope such estimators are also admissible among all estimators.
% \item WE SHOULD CITE \cite{Johnstone-2019}.
%\end{enumerate}

\appendix

 \section{Proofs of lemmas in Section 2}
% \section{Proofs of lemmas in Section \ref{sec:main.adm}}
\label{app.lemmas.sec}
 We take the following notation for \eqref{our.prior}, 
\begin{equation}\label{our.prior.2}
 \pi(\theta,\eta)
  =\eta^{-1}\times \eta^{p/2}\pi(\eta\|\theta\|^2\mymid a,b),
\end{equation}
where
\begin{equation}\label{pi.r.a.b}
 \begin{split}
&  \pi(r\mymid a,b) \\ 
 &=
  \int_0^1\frac{1}{(2\pi)^{p/2}}\left(\frac{\lambda}{1-\lambda}\right)^{p/2}
  \exp\left(-\frac{\lambda}{1-\lambda}\frac{r}{2}\right)
  \frac{\lambda^a(1-\lambda)^b}{B(a+1,b+1)}\rd \lambda.
 \end{split}
\end{equation} 
Recall
\begin{equation*}
f_{x}(x\mymid\theta,\eta)=\frac{\eta^{p/2}}{(2\pi)^{p/2}}\exp\left(-\frac{\eta}{2}\|x -\theta\|^2\right), \ f_s(s\mymid\eta)=\frac{\eta^{n/2}s^{n/2-1}}{\Gamma(n/2)2^{n/2}}\exp\left(-\frac{\eta s}{2}\right)
\end{equation*}
%\begin{equation*}
%\begin{split}
%& f_{x}(x\mymid\theta,\eta)=\frac{\eta^{p/2}}{(2\pi)^{p/2}}\exp\left(-\frac{\eta}{2}\|x -\theta\|^2\right), \\
%&f_s(s\mymid\eta)=\frac{\eta^{n/2}}{\Gamma(n/2)2^{n/2}}s^{n/2-1}\exp\left(-\frac{\eta s}{2}\right)
%\end{split}
%\end{equation*}
and
\begin{equation*}
 m(\psi)  = \int_{\mbR^p}\int_0^\infty 
\psi(\theta,\eta)f_{x }(x \mymid\theta,\eta)f_s(s\mymid\eta)
   \rd \theta \rd \eta,
\end{equation*}
where $ \psi(\theta,\eta)$ is possibly a vector function.

  \subsection{Proof of Lemma \ref{lem.Blyth}}
\label{sec.proof.lemma.1}
  Suppose that $ \delta_\pi$ is inadmissible and hence  $\delta'$ satisfies
  \begin{equation}\label{eq.weak}
R(\theta,\eta,\delta') \leq R(\theta,\eta,\delta_{\pi })
  \end{equation}
for all $(\theta,\eta)$ and
\begin{equation}\label{eq.strict}
R(\theta,\eta,\delta') < R(\theta,\eta,\delta_{\pi }) \text{ for some }(\theta_0,\eta_0).
\end{equation}
By \eqref{eq.strict}, we have
\begin{align*}
 \int_{\mbR^p}\int_0^\infty
 %\iint
 \| \delta_\pi(x,s)-\delta'(x,s)\|^2 f_{x }(x \mymid\theta_0,\eta_0)f_s(s\mymid\eta_0)\rd x\rd s>0.
\end{align*}
Further we have
\begin{align*}
 & %\iint
 \int_{\mbR^p}\int_0^\infty
 \| \delta_\pi(x,s)-\delta'(x,s)\|^2 f_{x }(x \mymid\theta,\eta)f_s(s\mymid\eta)\rd x\rd s \\
 &= %\iint
 \int_{\mbR^p}\int_0^\infty
 \| \delta_\pi(x,s)-\delta'(x,s)\|^2
 \frac{f_{x }(x \mymid\theta,\eta)f_s(s\mymid\eta)}{f_{x }(x \mymid\theta_0,\eta_0)f_s(s\mymid\eta_0)} f_{x }(x \mymid\theta_0,\eta_0)
 f_s(s\mymid\eta_0)
 \rd x\rd s. 
\end{align*}
Since the ratio 
\begin{align*}
 \frac{f_{x }(x \mymid\theta,\eta)f_s(s\mymid\eta)}{f_{x }(x \mymid\theta_0,\eta_0)f_s(s\mymid\eta_0)}
\end{align*}
is continuous in $(x,s)$ and positive, it follows that
\begin{align*}
 \int_{\mbR^p}\int_0^\infty
 %\iint
 \| \delta_\pi(x,s)-\delta'(x,s)\|^2 f_{x }(x \mymid\theta,\eta)f_s(s\mymid\eta)\rd x\rd s >0
\end{align*}
for all $(\theta,\eta)$.

Set $\delta''=(\delta_\pi+\delta')/2$. Then we have
\begin{align*}
 \|\delta''-\theta\|^2=\frac{\|\delta_\pi-\theta\|^2+\|\delta'-\theta\|^2}{2}
-\frac{\|\delta_\pi-\delta'\|^2}{4},
\end{align*}
and
\begin{align*}
R(\theta,\eta,\delta'') 
&=E\left[\eta\|\delta''-\theta\|^2\right] \\
&<(1/2)E\left[\eta\|\delta'-\theta\|^2\right] +(1/2)E\left[\eta\|\delta_\pi -\theta\|^2\right] \\
 &= \frac{1}{2}\left\{R(\theta,\eta,\delta')+R(\theta,\eta,\delta_\pi)\right\} \\
 &\leq  R(\theta,\eta,\delta_\pi),
\end{align*}
for all $(\theta,\eta)$.
Recall
\begin{align*}
 \Delta_i = \int_{\mbR^{p}}\int_0^\infty \left\{
R(\theta,\eta,\delta_{\pi })-R(\theta,\eta,\delta_{\pi i}) \right\}
\pi_i(\theta,\eta)
 \rd  \theta \rd \eta.
\end{align*}
Then we have
\begin{align*}
\Delta_i
 &
\geq \int_{\mbR^{p}}\int_0^\infty \left\{
R(\theta,\eta,\delta_{\pi })-R(\theta,\eta,\delta'') \right\}
\pi_i(\theta,\eta)
 \rd  \theta \rd \eta \\
&\geq \int_{\mbR^{p}}\int_0^\infty \left\{
R(\theta,\eta,\delta_{\pi })-R(\theta,\eta,\delta'') \right\}
\pi_1(\theta,\eta)
 \rd  \theta \rd \eta \\
&>0
\end{align*}
which contradicts $ \Delta_i \to 0$ as $i\to\infty$.

\subsection{Proof of Lemma \ref{lem.form.Bayes}}
\label{sec.proof.lemma.2}
Under the loss given by \eqref{loss},
the generalized Bayes estimator corresponding to $\pi (\theta,\eta)$ is
\begin{equation}\label{eq:bayes_expression}
\begin{split}
 \delta_\pi  (x ,s)
 &=\frac{%\iint
\int_{\mbR^p}\int_0^\infty 
\theta \eta f_{x }(x \mymid \theta,\eta)f_s(s\mymid \eta)\pi (\theta,\eta)\rd \theta \rd \eta}
 {%\iint
\int_{\mbR^p}\int_0^\infty 
 \eta f_{x }(x \mymid \theta,\eta)f_s(s\mymid \eta)\pi (\theta,\eta)\rd \theta \rd \eta}  \\
 &=x + \frac{%\iint
\int_{\mbR^p}\int_0^\infty 
(\theta-x ) \eta f_{x }(x \mymid \theta,\eta)f_s(s\mymid \eta)\pi (\theta,\eta)\rd \theta \rd \eta}
 {%\iint
\int_{\mbR^p}\int_0^\infty 
 \eta f_{x }(x \mymid \theta,\eta)f_s(s\mymid \eta)
\pi (\theta,\eta)\rd \theta \rd \eta} \\
% &=x + \frac{\iint
% %\int_{\mbR^p}\int_0^\infty 
%\{-\nabla_{\theta} f_{x }(x \mymid \theta,\eta)\}f_s(s\mymid \eta)\pi (\theta,\eta)\rd \theta \rd \eta}
% {\iint
% %\int_{\mbR^p}\int_0^\infty
% \eta f_{x }(x \mymid \theta,\eta)f_s(s\mymid \eta)
%\pi (\theta,\eta)\rd \theta \rd \eta} \\
 &=x + \frac{%\iint
 \int_{\mbR^p}\int_0^\infty 
\nabla_{\theta}\pi (\theta,\eta)f_{x }(x \mymid \theta,\eta)f_s(s\mymid \eta)
\rd \theta \rd \eta}
 {%\iint
 \int_{\mbR^p}\int_0^\infty 
\eta f_{x }(x \mymid \theta,\eta)f_s(s\mymid \eta)
\pi (\theta,\eta)\rd \theta \rd \eta} \\
&=x  + \frac{m(\nabla_{\theta} \pi )}{m(\pi \eta)},
\end{split}
\end{equation}
where $\nabla_{\theta}=(\partial/\partial \theta_1,\partial/\partial \theta_2,\dots ,\partial/\partial \theta_p )'$ and the third equality follows from the \cite{Stein-1974} identity.

\subsection{Proof of Lemma \ref{lem.Bayes.risk.diff}}
\label{sec.proof.lemma.3}
%Let $ \delta_\pi $ and $ \delta_{\pi i}$ be the generalized Bayes estimator corresponding $ \pi(\theta,\eta)$ and the proper Bayes estimator corresponding to $ \pi_i(\theta,\eta)$. 
%Then the Bayes risk difference of $ \delta_\pi $ and $ \delta_{\pi i}$ with respect to the non-scaled density $ \pi_i(\theta,\eta)$ is
\begin{equation}\label{eq:Bdifference}
\begin{split}
 \Delta_i &=
 \int_{\mbR^p}\int_0^\infty
 %\iint
 \left\{E\left[\eta\|\delta_{\pi }-\theta\|^2\right]-E\left[\eta\|\delta_{\pi i}-\theta\|^2\right] \right\}
\pi_i(\theta,\eta)
 \rd  \theta \rd \eta \\
% \int_{\mbR^{p}}\int_0^\infty \left\{
%R(\theta,\eta,\delta_{\pi })-R(\theta,\eta,\delta_{\pi i}) \right\}
%\pi_i(\theta,\eta)
% \rd  \theta \rd \eta\\
 &=
%\iint 
 \int_{\mbR^{p}}\int_0^\infty 
\left\{\int_{\mbR^p}\int_0^\infty \{\|\delta_\pi (x,s)-\theta\|^2-\|\delta_{\pi i}(x,s)-\theta\|^2\}
f_{x }(x \mymid\theta,\eta)f_s(s\mymid\eta)\rd x \rd s \right\}\\
&\qquad \times  \eta\pi_i(\theta,\eta) \rd  \theta \rd \eta \\
 &=%\iint
 \int_{\mbR^{p}}\int_0^\infty
 (\|\delta_{\pi }\|^2 - \|\delta_{\pi i} \|^2) m(\pi_i\eta) \rd x  \rd s
 -2
%\iint
  \int_{\mbR^{p}}\int_0^\infty
 m(\theta\eta \pi_i)'(\delta_\pi -\delta_{\pi i}) \rd x\rd s \\
 &=%\iint
\int_{\mbR^{p}}\int_0^\infty
 \| \delta_{\pi } - \delta_{\pi i} \|^2 m(\pi_i\eta) \rd x  \rd s \\
 &=%\iint
\int_{\mbR^{p}}\int_0^\infty
 \left\| \frac{m(\nabla \pi )}{m(\pi \eta)} - 
\frac{m(\nabla \{\pi h_i^2\})}{m(\pi h_i^2\eta)}
\right\|^2 m(\pi h_i^2\eta)\rd x  \rd s \\
 &=%\iint
\int_{\mbR^{p}}\int_0^\infty
 \left\| \frac{m(\nabla \pi )}{m(\pi \eta)} - 
\frac{m(h_i^2\nabla \pi )}{m(\pi h_i^2\eta)}
\right\|^2 m(\pi h_i^2\eta)\rd x  \rd s,
\end{split}
\end{equation}
where the fifth equality follows from \eqref{eq:bayes_expression}.

\subsection{Proof of Lemma \ref{lem.hi}}
\label{sec.proof.lemma.4}
The results follow from the integrals,
\begin{equation}\label{2i2i}
 \begin{split}
 \int_0^\infty \eta^{-1}h_i^2(\eta)\rd \eta&=
 \int_0^1 \frac{i^2\rd\eta}{\eta \{i+\log (1/\eta)\}^2}
+\int_1^\infty \frac{i^2\rd\eta}{\eta \{i+\log \eta\}^2}  \\
&= \left[\frac{i^2}{i+\log (1/\eta)}\right]_0^1+\left[-\frac{i^2}{i+\log \eta}\right]_1^\infty=
 2i
\end{split}
\end{equation}
and
\begin{equation}
 \begin{split}
  \int_{\mbR^p}\int_0^\infty
  %\iint
  \pi_i(\theta,\eta)\rd\theta\rd \eta &=
  \int_{\mbR^p}\int_0^\infty
  %\iint
  \pi(\theta,\eta)h_i^2(\eta) \rd\theta\rd \eta\\
&=\int_{\mbR^p} \eta^{p/2}\pi(\eta\|\theta\|^2\mymid a,b)\rd \theta\int_0^\infty\eta^{-1}h_i^2(\eta)\rd \eta \\
&=1\times 2i=2i.\end{split}
\end{equation}

\subsection{Proof of Lemma \ref{lem.integrand.converge}}
\label{sec.proof.lemma.5}
Recall $\delta_\pi$ is expressed as
\begin{align*}
\frac{\int_{\mbR^p}\int_0^\infty 
\theta \eta f_{x }(x \mymid \theta,\eta)f_s(s\mymid \eta)\pi(\theta,\eta)\rd \theta \rd \eta}
{\int_{\mbR^p}\int_0^\infty 
 \eta f_{x }(x \mymid \theta,\eta)f_s(s\mymid \eta)\pi(\theta,\eta)\rd \theta \rd \eta} 
\end{align*}
and $j$-th component is given by
\begin{align*}
\delta_{\pi,j}= \frac{\left(\int_{-\infty}^0+\int_0^\infty\right) \theta_j\left\{\int_{\mbR^{p-1}} \int_0^\infty 
 \eta f_{x }(x \mymid \theta,\eta)f_s(s\mymid \eta)\pi(\theta,\eta)\rd \theta_{-j}\rd \eta\right\}\rd\theta_j }
{\int_{\mbR^p}\int_0^\infty 
 \eta f_{x }(x \mymid \theta,\eta)f_s(s\mymid \eta)\pi(\theta,\eta)\rd \theta \rd \eta}, 
\end{align*}
where $\theta_{-j}=(\theta_1,\dots,\theta_{j-1},\theta_{j+1},\dots,\theta_p)\in\mathbb{R}^{p-1}$.
Due to the existence of $\delta_\pi$, the following three integrals exist and are finite for any $x$ and $s$,
\begin{align}
& \int_{-\infty}^0 (-\theta_j)\left\{\int_{\mbR^{p-1}} \int_0^\infty 
 \eta f_{x }(x \mymid \theta,\eta)f_s(s\mymid \eta)\pi(\theta,\eta)\rd \theta_{-j}\rd \eta\right\}\rd\theta_j, \label{eq:int.1}\\
&\int_0^\infty \theta_j\left\{\int_{\mbR^{p-1}}\int_0^\infty 
 \eta f_{x }(x \mymid \theta,\eta)f_s(s\mymid \eta)\pi(\theta,\eta)\rd \theta_{-j}\rd \eta\right\}\rd\theta_j,  \label{eq:int.2}\\
&\int_{\mbR^p}\int_0^\infty 
 \eta f_{x }(x \mymid \theta,\eta)f_s(s\mymid \eta)\pi(\theta,\eta)\rd \theta \rd \eta.\label{eq:int.3}
\end{align}
As in \eqref{eq:Bdifference}, the integrand of $\Delta_i$ can be written as
\begin{equation*}
 \| \delta_{\pi } - \delta_{\pi i} \|^2 m(\pi_i\eta).
\end{equation*}
By continuity of squared norm, %quadratic form, 
it suffices to show the $j$-th component of $\delta_{\pi i}$ approaches $ \delta_{\pi,j}$ as $i\to\infty$.
Note
\begin{align*}
 \delta_{\pi {i},j}
 &=
\frac{\int_{\mbR^p}\int_0^\infty 
\theta_j \eta f_{x }(x \mymid \theta,\eta)f_s(s\mymid \eta)\pi(\theta,\eta)h_i^2(\eta)\rd \theta \rd \eta}
{\int_{\mbR^p}\int_0^\infty 
 \eta f_{x }(x \mymid \theta,\eta)f_s(s\mymid \eta)\pi(\theta,\eta)h_i^2(\eta)\rd \theta \rd \eta} \\
 &=
\frac{\int_{\mbR^p}\int_0^\infty 
\theta_j \eta f_{x }(x \mymid \theta,\eta)f_s(s\mymid \eta)\pi(\theta,\eta)h_i^2(\eta)\rd \theta \rd \eta}
{\int_{\mbR^p}\int_0^\infty 
 \eta f_{x }(x \mymid \theta,\eta)f_s(s\mymid \eta)\pi(\theta,\eta)h_i^2(\eta)\rd \theta \rd \eta},
\end{align*}
where the numerator is decomposed as
\begin{align*}
& -\int_{-\infty}^0 (-\theta_j)\left\{\int_{\mbR^{p-1}} \int_0^\infty 
 \eta f_{x }(x \mymid \theta,\eta)f_s(s\mymid \eta)\pi(\theta,\eta)h_i^2(\eta)\rd \theta_{-j}\rd \eta\right\}\rd\theta_j \\
&\qquad + \int_0^\infty \theta_j\left\{\int_{\mbR^{p-1}} \int_0^\infty 
 \eta f_{x }(x \mymid \theta,\eta)f_s(s\mymid \eta)\pi(\theta,\eta)h_i^2(\eta)\rd \theta_{-j}\rd \eta\right\}\rd\theta_j.
\end{align*}
Note $h_i\leq 1$ and $\lim_{i\to\infty}h_i= 1$.
By the dominated convergence theorem, the two terms and the denominator converge the integrals \eqref{eq:int.1}, \eqref{eq:int.2}, \eqref{eq:int.3} respectively, which implies that the $j$-th component of $\delta_{\pi i}$ approaches $ \delta_{\pi,j}$ as $i\to\infty$.

\medskip
\begin{mycomment}
In Appendix \ref{sec.proof.lemma.6} -- \ref{sec:A}, \ref{sec:B} and \ref{sec.proof.lemma.2.8},
there are several positive constants denoted by $Q_i$, $C_i$ and $T_i$ (for $i=1,2,\dots,$).

% The first time such a $Q_i$ is used, all the arguments will be given. 
%After that, the arguments except for $p$, $n$, $a$ and $b$ will still be given.
%For example, we use $Q_1(p,a,b)$ the first time and $Q_1$ thereafter.
%We use $Q_5(p,n,a,b;i,j)$ the first time and $Q_5(i,j)$ thereafter.

The positive constants in Lemmas \ref{lem.bound.delta} -- \ref{lem.bound.B}, $C$, $D$, $E$ and $F$, $\gamma_1$ and $\gamma_2$, are expressed in terms of the $Q$'s and $C$'s as follows.
\begin{align*}
 C&=32Q_1^2Q_2 \\
 D&=Q_3Q_4 \\
 E&= 2C_6+2C_7\\
 F&= 2C_9+2C_{10}\\
 \gamma_1&= \min\{C_5, 1/\exp(4C_4(1))\} \\
\gamma_2&= \max\{C_8,\exp(4C_4(1))\}.
\end{align*}
\end{mycomment}
 
\subsection{Proof of Lemma \ref{lem.bound.delta}}
\label{sec.proof.lemma.6}
Assume $b>-1/2$.
By Part \ref{lem:pirab.1} of Lemma \ref{lem:pirab}, there exists $Q_1>0$ such that
\begin{equation}\label{Q1}
 \begin{split}
 \|\theta\|\frac{\|\nabla_\theta \pi (\theta,\eta)\|}{\pi (\theta,\eta)}
   &=\|\theta\|\frac{\|2\eta\theta (\rd /\rd r)\pi(r\mymid a,b)|_{r=\eta\|\theta\|^2}\|\eta^{p/2-1}}
{\pi(\eta\|\theta\|^2\mymid a,b)\eta^{p/2-1}} \\
  &=2 \left\{r\frac{|\pi'(r)|}{\pi(r)}\right\}|_{r=\eta\|\theta\|^2}  <2 Q_1,
 \end{split}
\end{equation} 
for any $\theta$ and $\eta$.

In the integrand of \eqref{eq:Bdifference}, we have
\begin{equation}\label{Q2}
 \begin{split}
& \left| \frac{1}{m(\pi \eta)} - \frac{h_i^2}{m(\pi h_i^2\eta)}\right|
  \\
& =\left( \frac{1}{\sqrt{m(\pi \eta)}} + \frac{h_i}{\sqrt{m(\pi h_i^2\eta)}}\right)
\left| \frac{1}{\sqrt{m(\pi \eta)}} - \frac{h_i}{\sqrt{m(\pi h_i^2\eta)}}\right| \\
& =\frac{1}{\sqrt{m(\pi \eta)m(\pi h_i^2\eta)}}
\left( \frac{\sqrt{m(\pi h_i^2\eta)}}{\sqrt{m(\pi \eta)}} + h_i\right)
\left| 1 - \frac{\sqrt{m(\pi \eta)}}{\sqrt{m(\pi h_i^2\eta)}}h_i\right| \\
& \leq\frac{2}{\sqrt{m(\pi \eta)m(\pi h_i^2\eta)}}
\left| 1 - \frac{\sqrt{m(\pi \eta)}}{\sqrt{m(\pi h_i^2\eta)}}h_i\right|, 
 \end{split}
\end{equation}
where the inequality follows from the fact $0\leq h_i\leq 1$.
Then
\begin{align}
& \left\| \frac{m(\nabla \pi )}{m(\pi \eta)} - \frac{m(h_i^2\nabla \pi )}{m(\pi h_i^2\eta)}
\right\|^2 m(\pi h_i^2\eta) \notag \\
& = \left\| m\left(\nabla \pi  \left(\frac{1}{m(\pi \eta)} - 
\frac{h_i^2}{m(\pi h_i^2\eta)}\right)\right)
\right\|^2 m(\pi h_i^2\eta) \notag \\
& \leq \left\{ m\left(\left\|\nabla \pi  \right\|\left|\frac{1}{m(\pi \eta)} - 
\frac{h_i^2}{m(\pi h_i^2\eta)}\right|\right)
\right\}^2 m(\pi h_i^2\eta) \notag \\
& \leq 4Q_1^2\left\{ m\left(\frac{\pi }{\|\theta\|}\left|\frac{1}{m(\pi \eta)} - 
\frac{h_i^2}{m(\pi h_i^2\eta)}\right|\right)
\right\}^2 m(\pi h_i^2\eta) \notag \\
& \leq  \frac{16Q_1^2}{m(\pi \eta)}\left\{
m\left(\frac{\pi }{\|\theta\|}
\left| 1 - \frac{\sqrt{m(\pi \eta)}}{\sqrt{m(\pi h_i^2\eta)}}h_i\right|
\right)\right\}^2  \notag \\
& \leq  \frac{16Q_1^2}{m(\pi \eta)}
 m\left(\frac{\pi }{\eta\|\theta\|^2}k(\eta\|\theta\|^2)\right)
m\left(\frac{\pi \eta}{k(\eta\|\theta\|^2)} \left\{1 - \frac{\sqrt{m(\pi \eta)}}{\sqrt{m(\pi h_i^2\eta)}}h_i\right\}^2\right),
  \notag 
\end{align}
where the second, third and fourth inequality follow from \eqref{Q1}, \eqref{Q2} and Cauchy-Schwarz inequality and
\begin{align*}
 k(r)=r^{1/2}I_{[0,1]}(r)+I_{(1,\infty)}(r).
\end{align*}
By Lemma \ref{lem.ASMS}, there exists a constant $Q_2>1$ such that
\begin{align*}
 m\left(\frac{\pi \eta}{k(\eta\|\theta\|^2)} \left\{1 - \frac{\sqrt{m(\pi \eta)}}{\sqrt{m(\pi h_i^2\eta)}}h_i\right\}^2\right) 
\leq Q_2 m\left(\pi \eta \left\{1 - \frac{\sqrt{m(\pi \eta)}}{\sqrt{m(\pi h_i^2\eta)}}h_i\right\}^2\right).
\end{align*}
Further we have
\begin{align*}
 m\left(\pi \eta \left\{1 - \frac{\sqrt{m(\pi \eta)}}{\sqrt{m(\pi h_i^2\eta)}}h_i\right\}^2\right) 
 & =2m(\pi \eta)\left\{1- \frac{m(\pi \eta h_i)}{m(\pi \eta)}\frac{\sqrt{m(\pi \eta)}}{\sqrt{m(\pi h_i^2\eta)}}\right\} \\
& =2m(\pi \eta)\left\{1- \sqrt{\frac{m(\pi \eta h_i)^2}{m(\pi \eta)m(\pi h_i^2\eta)}}\right\} \\
& \leq 2m(\pi \eta)\left\{1- \frac{m(\pi \eta h_i)^2}{m(\pi \eta)m(\pi h_i^2\eta)}\right\},
\end{align*}
where the inequality follows from the fact
\begin{align*}
 \frac{m(\pi \eta h_i)^2}{m(\pi \eta)m(\pi h_i^2\eta)}\in(0,1)
\end{align*}
which is shown by Cauchy-Schwarz inequality.
Hence we have
\begin{equation}\label{ABAB}
 \begin{split}
& \left\| \frac{m(\nabla \pi )}{m(\pi \eta)} - 
\frac{m(h_i^2\nabla \pi )}{m(\pi h_i^2\eta)}
\right\|^2 m(\pi h_i^2\eta) \\
&\leq
32Q_1^2Q_2
m\left(\frac{\pi }{\eta\|\theta\|^2}k(\eta\|\theta\|^2)\right)
  \left\{1- \frac{m(\pi \eta h_i)^2}{m(\pi \eta)m(\pi h_i^2\eta)}\right\} \\
&=32Q_1^2Q_2 A(x,s)B_i(x,s)
\end{split}
\end{equation}
and, by \eqref{eq:Bdifference},
\begin{equation}\label{Delta.upper}
\Delta_i \leq  32Q_1^2Q_2 \int_{\mbR^p}\int_0^\infty 
 A(x,s)B_i(x,s)
 \rd x\rd s.
\end{equation}
Hence with $C=32Q_1^2Q_2$, the proof follows.

\medskip

In Sub-Sections \ref{sec:A} and \ref{sec:B}, we bound $A(x,s)$ and $B_i (x,s)$ from above, respectively.

\subsection{Proof of Lemma \ref{lem.bound.A}}
\label{sec:A}
By Part \ref{lem:pirab.2} of lemma \ref{lem:pirab}, there exists $Q_3>0$ such that
\begin{equation}
 \begin{split}\label{eq.A.0}
\frac{\pi(\theta,\eta) }{\eta\|\theta\|^2}k(\eta\|\theta\|^2) 
 &= \eta^{p/2-1}\pi(\eta\|\theta\|^2\mymid a,b)
 \left\{\frac{I_{[0,1]}(\eta\|\theta\|^2)}{\eta^{1/2}\|\theta\|}+
 \frac{I_{(1,\infty)}(\eta\|\theta\|^2)}{\eta\|\theta\|^2} \right\}\\
&\leq Q_3\eta^{p/2-1}\pi(\eta\|\theta\|^2\mymid a+1,\tilde{b})
\end{split}
\end{equation}
where
\begin{align*}
 \tilde{b}=\begin{cases}-1/2 & b\geq 0\\
	  b-1/2 & -1/2<b<0.
	 \end{cases}
\end{align*}
Let
\begin{equation}\label{eq.A.1}
 \pi_*(\theta,\eta)=\eta^{p/2-1}\pi(\eta\|\theta\|^2\mymid a+1,\tilde{b}).
\end{equation}
Note
\begin{equation}\label{qf}
 \frac{\lambda}{1-\lambda}\|\theta\|^2+\|x-\theta\|^2=\frac{1}{1-\lambda}\|\theta-(1-\lambda)x\|^2+\lambda\|x\|^2.
\end{equation}
Integrating w.r.t.~$\theta$ and $\eta$, we have
\begin{align*}
m(\pi_*) &=\frac{s^{n/2-1}}{c_{p,n} } \int_0^1\int_0^\infty \eta^{(p+n)/2-1}\exp\left(-\eta\frac{\lambda\|x\|^2+s}{2}\right)
 \frac{\lambda^{p/2+a+1}(1-\lambda)^{\tilde{b}}}{B(a+2,\tilde{b}+1)}\rd \lambda \rd \eta \\
 &=s^{-p/2-1}\frac{2^{(p+n)/2}\Gamma((p+n)/2)}{c_{p,n}  B(a+2,\tilde{b}+1)} \int_0^1 
  \frac{\lambda^{p/2+a+1}(1-\lambda)^{\tilde{b}}}{(1+\lambda \|x\|^2/s)^{(p+n)/2}}\rd \lambda,
\end{align*}
where 
\begin{align}\label{eq:C1}
c_{p,n} = (2\pi)^{p/2}\Gamma(n/2)2^{n/2}.
\end{align}
By the change of variables
\begin{equation}\label{cv.1}
 t=\frac{(1+w)\lambda}{1+\lambda w},\ \frac{\rd \lambda}{\rd t}=\frac{1}{1+w}\frac{1}{(1-\{w/(1+w)\}t)^2}, 
\end{equation}
where $w=\|x\|^2/s $, we have
\begin{align*}
 \lambda=(1-z)\frac{t}{1-zt}, \ 1-\lambda=\frac{1-t}{1-zt}, \ 1+\lambda w=\frac{1}{1-zt},
\end{align*}
where $z=w/(1+w)$, and hence
\begin{align*}
m(\pi_*) 
 &=\frac{(1-z)^{p/2+a+2}}{s^{p/2+1}}
 \frac{2^{(p+n)/2}\Gamma((p+n)/2)}{c_{p,n}  B(a+2,\tilde{b}+1)}
 \int_0^1 t^{p/2+a+1}(1-t)^{\tilde{b}}(1-zt)^{n/2-a-\tilde{b}-3}\rd t.
\end{align*}
Since $-a-3/2<-(a+1)<0$, we have
\begin{align*}
 (1-zt)^{-a-3/2} \leq (1-z)^{-a-3/2} \ \text{for all} \ t\in(0,1)
\end{align*}
and hence
\begin{equation*}
m(\pi_*) 
\leq \frac{(1-z)^{p/2+1/2}}{s^{p/2+1}}
 \frac{2^{(p+n)/2}\Gamma((p+n)/2)}{c_{p,n}  B(a+2,\tilde{b}+1)}
 \int_0^1 t^{p/2+a+1}(1-t)^{\tilde{b}}(1-zt)^{n/2-\tilde{b}-3/2}\rd t. 
\end{equation*}
Further, since $ 1-zt$ is monotone in $z$, we have
\begin{equation}\label{eq.A.2}
 m(\pi_*) \leq Q_4 s^{-p/2-1}(1-z)^{p/2+1/2},
\end{equation}
where
\begin{align*}
 Q_4=\frac{2^{(p+n)/2}\Gamma((p+n)/2)}{c_{p,n}  B(a+2,\tilde{b}+1)}
 \max\left\{B(p/2+a+2,\tilde{b}+1), B(p/2+a+2,n/2-1/2)\right\}.
\end{align*}
Therefore,
by \eqref{eq.A.0}, \eqref{eq.A.1} and \eqref{eq.A.2}, we have
\begin{equation}\label{eq.A.important}
A(x,s)= m\left(\frac{\pi(\theta,\eta) }{\eta\|\theta\|^2}k(\eta\|\theta\|^2)\right)
\leq \frac{D}{ s^{p/2+1}}\left(\frac{1}{1+\|x\|^2/s}\right)^{p/2+1/2}
\end{equation}
where $D=Q_3Q_4$, completing the proof.

\subsection{Proof of Lemma \ref{lem.bound.B}}
\label{sec:B}
The proof is based on Lemmas \ref{lem.2.8.1} -- \ref{lem.2.8.3}, whose proofs are given in Subsection \ref{sec.proof.lemma.2.8}.
First we re-express $B_i (x,s)$ as follows.
\begin{lemma}\label{lem.2.8.1}
\begin{equation}\label{BBB}
\begin{split}
B_i (x,s)= 1-\frac{\{E[H_i(V/s)\mymid z]\}^2}{E[H_i^2(V/s)\mymid z]},
\end{split} 
\end{equation}
 where the expected value is with respect to the probability density on $v\in(0,\infty)$,
\begin{align}
 f(v\mymid z)&=\frac{v^{(p+n)/2}}{\psi(z)}
\int_0^1 \frac{t^{p/2+a}(1-t)^b}{(1-zt)^{p/2+a+b+2}}\exp\left(-\frac{v}{2(1-zt)}\right) \rd t,\label{fff}
\end{align} 
with normalizing constant $\psi(z)$ given below in \eqref{psipsi} and
 \begin{equation}\label{HIHI}
H_i(\eta)= \frac{h_i(\eta)}{i}=\frac{1}{i+\log (\max(\eta,1/\eta))}.
 \end{equation}
\end{lemma}
The behavior of the probability density $f$ given in \eqref{fff} is summarized in the following lemma.
\begin{lemma}\label{lem.2.8.2}
 Suppose $n/2-a>0$.
 \begin{enumerate}
  \item \label{lem.2.8.2.1}
       For $s\leq 1$ and for $k\geq 0$, there exist $C_1(k)>0$ and $C_2(k)>0$ such that
 \begin{align}\label{eq.lem:small.s.log.new}
s^{-C_1(k)}\int_0^s  |\log v|^k f(v\mymid z)\rd v\leq C_2(k).
 \end{align}
%\begin{align*}
%      q(k)&=\begin{cases}
% (n/2-a)(4-\min\{1,2\frac{b+1}{n/2-a}\})/4 & k=0, \\
%(n/2-a)(3-\min\{1,2\frac{b+1}{n/2-a}\})/4  & k>0, 
%	   \end{cases} \\
%      Q_8(k)&=\begin{cases}
%Q_7\left(\min\{\frac{n/2-a}{b+1},2\}/4\right) & k=0,
%	   \\
%	      \left(\frac{4k}{n/2-a}\right)^k
%Q_7\left(\min\{\frac{n/2-a}{b+1},2\}/4\right)	      & k>0.
%	  \end{cases} 
%\end{align*}
  \item \label{lem.2.8.2.2}
	For $s> 1$ and for $k\geq 0$, there exists $C_3(k)>0$ such that
 \begin{align*}
\exp\left(s/4\right)  \int_s^\infty  |\log v|^k f(v\mymid z)\rd v \leq  C_3(k).
 \end{align*}
% \begin{align*}
%  Q_{10}(k)=
%\begin{cases}
%T_2(0,2)/Q_5(0,1) & k=0, \\
% \left(\frac{2k}{n/2-a}\right)^k\{T_2(-1/2,2)+T_2(1/2,2)\}/Q_5(0,1) & k>0.
%\end{cases}
% \end{align*}
 \end{enumerate} 
\end{lemma}
It follows from Lemma \ref{lem.2.8.2} that
\begin{equation}\label{lem.2.8.futou.0}
 E\left[|\log V|^k\mymid z\right]<C_2(k)+C_3(k)\coloneqq C_4(k).
\end{equation}
Using Lemma \ref{lem.2.8.2}, $\{E[H_i(V/s)\mymid z]\}^2$ and $E[H_i^2(V/s)\mymid z]$ with
$H_i(\cdot)$ given in \eqref{HIHI} are bounded as follows.
\begin{lemma}\label{lem.2.8.3}
 \begin{enumerate}
  \item \label{lem.2.8.3.1}
	There exist $0<C_5<1$ and $C_6>0$ such that
\begin{equation}
\{i+\log(1/s)\}^2\{E[H_i(V/s)\mymid z]\}^2 
\geq 1-2\frac{E\left[\log V\mymid z\right]}{i+\log(1/s)}-\frac{C_6}{\{1+\log(1/s)\}^2}
\end{equation}
for all $0<s<C_5$, all $z\in(0,1)$ and all positive integers $i$.
  \item \label{lem.2.8.3.2}
	There exists $C_7>0$ such that
\begin{equation}
\{i+\log(1/s)\}^2E[H_i^2(V/s)\mymid z]
\leq 1-2\frac{E\left[\log V\mymid z\right]}{i+\log(1/s)}+\frac{C_7}{\{1+\log(1/s)\}^2}
\end{equation}
for all $0<s<1$, all $z\in(0,1)$ and all positive integers $i$.
  \item \label{lem.2.8.3.3}
	There exist $C_8>1$ and $C_9>0$ such that
\begin{equation}
(i+\log s)^2\{E[H_i(V/s)\mymid z]\}^2 
\geq 1+2\frac{E\left[\log V\mymid z\right]}{i+\log s}-\frac{C_9}{(1+\log s)^2}
\end{equation}
for all $s>C_8$, all $z\in(0,1)$ and all positive integers $i$.
  \item \label{lem.2.8.3.4}
	There exists $C_{10}>0$ such that
\begin{equation}
(i+\log s)^2 E[H_i^2(V/s)\mymid z]
\leq 1+2\frac{E\left[\log V\mymid z\right]}{i+\log s}+\frac{C_{10}}{(1+\log s)^2}
\end{equation}
for all $s>1$, all $z\in(0,1)$ and all positive integers $i$.
 \end{enumerate}
\end{lemma}
%Note, by Lemma \ref{lem.2.8.2}, with $k=1$
%\begin{align}\label
% E\left[|\log V|\mymid z\right]<C_2(1)+C_3(1).
%\end{align}
Using Lemmas \ref{lem.2.8.1} -- \ref{lem.2.8.3}, we now complete the proof in the subintervals,
$(0,\gamma_1)$, $(\gamma_2,\infty)$ and $[\gamma_1,\gamma_2]$, respectively.

\smallskip

[\textbf{CASE \ I}] \ 
We first bound $B_i(x,s)$ for $0<s<\gamma_1$ where $\gamma_1$ is defined by
\begin{equation}
 \gamma_1=\min\{C_5, 1/\exp(4C_4(1))\}.
\end{equation}
Note, for $0<s<\gamma_1$, 
\begin{equation}\label{lem.2.8.futou}
 1- 2\frac{E\left[\log V\mymid z\right]}{i+\log(1/s)}
\geq 1-2\frac{E\left[|\log V|\mymid z\right]}{\log(1/s)}\geq 1-\frac{2C_4(1)}{4C_4(1)}
=\frac{1}{2},
\end{equation}
where the second inequality follows from \eqref{lem.2.8.futou.0}.
Further, by Parts \ref{lem.2.8.3.1} and \ref{lem.2.8.3.2} of Lemma \ref{lem.2.8.3},
for $0<s<\gamma_1$, 
we have
\begin{equation}
\begin{split}
 B_i (x,s)&= 1-\frac{\{E[H_i(V/s)\mymid z]\}^2}{E[H_i^2(V/s)\mymid z]} \\
 &\leq 1-\frac{1-2\frac{E\left[\log V\mymid z\right]}{i+\log(1/s)}-\frac{C_6}{\{1+\log(1/s)\}^2}}
 {1-2\frac{E\left[\log V\mymid z\right]}{i+\log(1/s)}+\frac{C_7}{\{1+\log(1/s)\}^2}} \\
 &= \frac{C_6+C_7}{\{1+\log(1/s)\}^2}
\frac{1}{1-2\frac{E\left[\log V\mymid z\right]}{i+\log(1/s)}+\frac{C_7}{\{1+\log(1/s)\}^2}} \\
 &\leq \frac{2C_6+2C_7}{\{1+\log(1/s)\}^2}
\end{split} 
\end{equation}
where the second inequality follows from \eqref{lem.2.8.futou}.
With $E=2C_6+2C_7$, the proof for $0<s<\gamma_1$ is complete.

\smallskip

[\textbf{CASE \ II}] \ 
Here we bound $B_i(x,s)$ for $s>\gamma_2>1$ where $\gamma_2$ is defined by
\begin{equation}
 \gamma_2=\max\{C_8,\exp(4C_4(1))\}.
\end{equation}
Note, for $s>\gamma_2$, 
\begin{equation}\label{lem.2.8.futou.2}
 1+ 2\frac{E\left[\log V\mymid z\right]}{i+\log s}
\geq 1-2\frac{E\left[|\log V|\mymid z\right]}{\log s}\geq 1-\frac{2C_4(1)}{4C_4(1)}
=\frac{1}{2},
\end{equation}
where the second inequality follows from \eqref{lem.2.8.futou.0}.
Further, by Parts \ref{lem.2.8.3.3} and \ref{lem.2.8.3.4} of Lemma \ref{lem.2.8.3},
for $s>\gamma_2$, 
we have
\begin{equation}
\begin{split}
 B_i (x,s)&= 1-\frac{\{E[H_i(V/s)\mymid z]\}^2}{E[H_i^2(V/s)\mymid z]} \\
 &\leq 1+\frac{1+2\frac{E\left[\log V\mymid z\right]}{i+\log s}-\frac{C_9}{(1+\log s)^2}}
 {1+2\frac{E\left[\log V\mymid z\right]}{i+\log s}+\frac{C_{10}}{(1+\log s)^2}} \\
 &= \frac{C_9+C_{10}}{(1+\log s)^2}
\frac{1}{1+2\frac{E\left[\log V\mymid z\right]}{i+\log s}+\frac{C_{10}}{(1+\log s)^2}} \\
 &\leq \frac{2C_9+2C_{10}}{(1+\log s)^2}
\end{split} 
\end{equation}
where the second inequality follows from \eqref{lem.2.8.futou.2}.
With $F=2C_9+2C_{10}$, the proof for $s>\gamma_2$ is complete.

Also, by \eqref{ABI}, $B_i\leq 1$ for all $x$ and $s$ and thus the bound for $\gamma_1\leq s \leq \gamma_2$ is $1$. This completes the proof.
  \subsection{Proof of Lemmas in Appendix \ref{sec:B}}
\label{sec.proof.lemma.2.8}
\subsubsection{Proof of Lemma \ref{lem.2.8.1}}
For 
\begin{align*}
 B_i (x,s)= 1- \frac{m(\pi \eta h_i)^2}{m(\pi \eta)m(\pi h_i^2\eta)},
\end{align*}
we re-express $m(\pi \eta L)$ with 
\begin{align*}
 L(\eta) =1, \ h_i(\eta), \text{ and } h_i^2(\eta).
\end{align*}
Noting \eqref{qf} and integrating w.r.t.~$\theta$, we have
\begin{align*}
m(\pi\eta L ) =\frac{s^{n/2-1}}{c_{p,n} } \int_0^1\int_0^\infty  \eta^{(p+n)/2}L (\eta)\exp\left(-\eta\frac{\lambda\|x\|^2+s}{2}\right)
\frac{\lambda^{p/2+a}(1-\lambda)^b}{B(a+1,b+1)}\rd \lambda \rd \eta,
\end{align*}
where $c_{p,n} $ is given by \eqref{eq:C1}.
By the change of variables given in \eqref{cv.1},
we have
\begin{align*}
 m(\pi \eta L )&=\frac{s^{n/2-1}(1-z)^{p/2+a+1}}{c_{p,n} B(a+1,b+1)}
\int_0^\infty\int_0^1 \frac{t^{p/2+a}(1-t)^b}{(1-zt)^{p/2+a+b+2}} \\
&\quad\times \eta^{(p+n)/2}L (\eta)\exp\left(-\frac{\eta}{1-zt}\frac{s}{2}\right)\rd \eta \rd t,
\end{align*}
where $z=w/(w+1)=\|x\|^2/(\|x\|^2+s)$.
Further, by the change of variables, $ v=\eta s$, we have
\begin{align*}
m(\pi \eta L )&=%\frac{1}{(2\pi)^{p/2}}\frac{s^{n/2-1}}{\Gamma(n/2)2^{n/2}}
\frac{s^{-p/2-2}(1-z)^{p/2+a+1}}{c_{p,n} B(a+1,b+1)}
% \frac{(1-z)^{p/2+a+1}}{c_1s^{p/2+2}}
\int_0^\infty\int_0^1  \frac{t^{p/2+a}(1-t)^b}{(1-zt)^{p/2+a+b+2}} \\
&\qquad\times v^{(p+n)/2}L (v/s)\exp\left(-\frac{v}{2(1-zt)}\right)\rd v \rd t \\
&=\frac{s^{-p/2-2}(1-z)^{p/2+a+1}}{c_{p,n} B(a+1,b+1)}\psi(z)
\int_0^\infty L (v/s)f(v\mymid z)\rd v \\
&=\frac{s^{-p/2-2}(1-z)^{p/2+a+1}}{c_{p,n} B(a+1,b+1)}\psi(z)E[L (V/s)\mymid z]
 % \frac{(1-z)^{p/2+a+1}}{c_1s^{p/2+2}}
\end{align*}
where $\psi(z)$ is the normalizing constant given by
\begin{align}\label{psipsi} 
  \psi(z) = \int_0^\infty\int_0^1 \frac{t^{p/2+a}(1-t)^b}{(1-zt)^{p/2+a+b+2}}v^{(p+n)/2}\exp\left(-\frac{v}{2(1-zt)}\right)\rd v \rd t.
\end{align}
The the result follows.
  \subsubsection{Properties of $\psi(z) $ and $f(v\mymid z)$}
\label{sec.prop.psi.f}
We present preliminary results for Lemma \ref{lem.2.8.2}.
We consider a function more general than $\psi(z)$ given by \eqref{psipsi}. Let
\begin{equation}\label{psipsi.1}
\begin{split}
 \psi(z;j,k)  &= \int_0^\infty\int_0^1 \frac{t^{p/2+a}(1-t)^b}{(1-zt)^{p/2+a+b+2}}v^{(p+n)/2+(n/2-a)j} \\ &\qquad\times \exp\left(-\frac{v}{2k(1-zt)}\right)\rd v \rd t,
\end{split}  
\end{equation}
under the conditions
\begin{align}\label{Asss}
 n/2-a>0,\quad j>-1, \quad k>0.
\end{align}
Clearly $\psi(z)$ given by \eqref{psipsi} is
\begin{align*}
 \psi(z)=\psi(z;0,1).
\end{align*}
\begin{lemma}\label{lem:psi}
 Assume the assumption \eqref{Asss}. Then
\begin{equation}\label{psipsipsi.1}
 0<\psi(0;j,k)<\infty\text{ and }0<\psi(1;j,k)<\infty.
\end{equation}
Further
 \begin{equation}\label{c2c2}
T_1(j,k) \leq \psi(z;j,k)\leq T_2(j,k),
%T_1(p,n,a,b;j,k) \leq \psi(z;j,k)\leq T_2(p,n,a,b;j,k),
% \min\{\psi(0;i,k),\psi(1;i,k)\} \leq \psi(z;i,k)\leq \max\{\psi(0;i,k),\psi(1;i,k)\}.
 \end{equation}
 where
\begin{align*}
 T_1(j,k)&=\min\{\psi(0;j,k),\psi(1;j,k)\}\\ \text{and } \ 
  T_2(j,k)&=\max\{\psi(0;j,k),\psi(1;j,k)\}.
\end{align*} 
\end{lemma}
\begin{proof}
Note
\begin{equation}\label{c2}
 \begin{split}
\psi(z;j,k) 
  &=\Gamma((p+n)/2+1+(n/2-a)j)\{2k\}^{(p+n)/2+1+(n/2-a)j}\\
  &\quad\times\int_0^1 t^{p/2+a}(1-t)^b(1-zt)^{(n/2-a)(j+1)-b-1} \rd t, 
\end{split}
\end{equation}
which is monotone in $z$ (either increasing or decreasing depending on the sign of $(n/2-a)(j+1)-b-1$). Further we have
\begin{align*}
 \psi(0;j,k)&=\Gamma((p+n)/2+1+(n/2-a)j)\{2k\}^{(p+n)/2+1+(n/2-a)j}\\ &\qquad\times B(p/2+a+1,b+1), \\
 \psi(1;j,k)&=\Gamma((p+n)/2+1+(n/2-a)j)\{2k\}^{(p+n)/2+1+(n/2-a)j}\\ &\qquad\times B(p/2+a+1,(n/2-a)(j+1)),
\end{align*}
which are positive and finite under the assumption \eqref{Asss}.
Thus \eqref{psipsipsi.1} and \eqref{c2c2} follow.
\end{proof}

 \begin{lemma}\label{lem.f.small}
For any $\epsilon\in(0,1)$ and $v\in(0,1)$,
 \begin{align*}
  f(v\mymid z)\leq
  T_3(\epsilon)
%  T_3(p,n,a,b;\epsilon)
  v^{n/2-a-1-\epsilon(b+1)}
 \end{align*}
 where
\begin{align*}
T_3(\epsilon)
 =\frac{\{p+2a+2+2\epsilon(b+1)\}^{p/2+a+1+\epsilon(b+1)}B(p/2+a+1,(b+1)\epsilon)}{T_1(0,1)}.
\end{align*}
  \end{lemma}
\begin{proof}
  Note, for $\epsilon\in(0,1)$,
  \begin{align*}
& (1-t)^b(1-zt)^{-p/2-a-b-2} \\
&= (1-t)^{(b+1)\epsilon-1}\left(\frac{1-t}{1-zt}\right)^{(b+1)(1-\epsilon)}
   (1-zt)^{-p/2-a-1-\epsilon(b+1)} \\
& \leq  (1-t)^{(b+1)\epsilon-1}(1-zt)^{-p/2-a-1-\epsilon(b+1)}.
  \end{align*}
 Also note for $v\in(0,1)$, since $a>-1$ and $b>-1$,
\begin{align*}
 v/2<1/2\leq p/2<p/2+a+1+\epsilon(b+1).
\end{align*}
Then, by Lemma \ref{lem:polyexp}, we have
  \begin{align*}
& (1-zt)^{-p/2-a-1-\epsilon(b+1)}
\exp\left(-\frac{v}{2(1-zt)}\right) \\
&\leq \left(\frac{p+2a+2+2\epsilon(b+1)}{v}\right)^{p/2+a+1+\epsilon(b+1)}.
  \end{align*}
Further, by Lemma \ref{lem:psi}, $\psi(z)\leq T_1(0,1)$ for all $z\in(0,1)$.
 Hence,
 by the definition of $f(v\mymid z)$ given by \eqref{fff}, %and the definition $\psi$ given by  \eqref{c2},
 we have
\begin{align*}
 f(v\mymid z)\leq T_3(\epsilon)v^{n/2-a-1-\epsilon(b+1)}.
\end{align*}  
\end{proof}

   \subsubsection{Proof of Lemma \ref{lem.2.8.2}}
\mbox{}

   [Part \ref{lem.2.8.2.1}] \
 Let
 \begin{align*}
  \epsilon_*=\frac{1}{4}\min\left(\frac{n/2-a}{b+1},2\right)\in(0,1)
 \end{align*}
 in Lemma \ref{lem.f.small}.
 Then we have
\begin{align*}
 C_1(0)\coloneqq n/2-a-\epsilon_*(b+1)=\frac{n/2-a}{4}\left(4-\min\left\{1,2\frac{b+1}{n/2-a}\right\}\right)>0
\end{align*} 
 and hence
\begin{equation}\label{Q8Q8.1}
 \int_0^s  f(v\mymid z)\rd v\leq \frac{T_3(\epsilon_*)}{C_1(0)}s^{C_1(0)}=C_2(0)s^{C_1(0)},
\end{equation}
 where $C_2(0)$ is defined by $C_2(0)=T_3(\epsilon_*)/C_1(0)$.
%\begin{align*}
%=n/2-a-\epsilon_*(b+1)
% %=\frac{n/2-a}{4}\left(4-\min\left\{1,2\frac{b+1}{n/2-a}\right\}\right)
% ,\  C_2(0)=\frac{T_3(\epsilon_*)}{C_1(0)}.
%\end{align*}

 For $k>0$, by Part \ref{lem:log.poly.1} of Lemma \ref{lem:log.poly}, we have
\begin{align*}
| \log v|^k\leq \left(\frac{4k}{n/2-a}\right)^kv^{-(n/2-a)/4}.
\end{align*}
For $k>0$, we have
\begin{align*}
 C_1(k)&\coloneqq n/2-a-\epsilon_*(b+1)-\frac{1}{4}(n/2-a) \\ &=(n/2-a)\left(\frac{3}{4}-\frac{1}{4}\min\left\{1,2\frac{b+1}{n/2-a}\right\}\right)>0.
\end{align*} 
 Then, for $k>0$,
 \begin{equation}\label{Q8Q8.2}
 \int_0^s  | \log v|^k f(v\mymid z)\rd v\leq C_2(k) s^{C_1(k)},
\end{equation}
 where $C_2(k)$ is defined by
\begin{align*}
 C_2(k)=\frac{T_3(\epsilon_*)}{C_1(k)}\left(\frac{4k}{n/2-a}\right)^k.
\end{align*}
% $ C_2(k)=T_3(\epsilon_*)\{(4k)/(n/2-a)\}^k/C_1(k)$.
 By \eqref{Q8Q8.1} and \eqref{Q8Q8.2}, Part \ref{lem.2.8.2.1} follows.

   [Part \ref{lem.2.8.2.2}] \
%By Part \ref{lem:log.poly.2} of Lemma \ref{lem:log.poly},  
 Note, for $v\geq s$,
\begin{align*}
 \exp\left(-\frac{v}{2(1-zt)}\right) 
&= \exp\left(-\frac{v}{4(1-zt)}\right)\exp\left(-\frac{v}{4(1-zt)}\right)\\
&\leq \exp\left(-\frac{v}{4(1-zt)}\right)\exp\left(-\frac{v}{4}\right) \\
&\leq \exp\left(-\frac{v}{4(1-zt)}\right)\exp\left(-\frac{s}{4}\right) .
\end{align*} 
 For $k=0$, by Lemma \ref{lem:psi}, we have
\begin{align*}
& \exp(s/4)\int_s^\infty  f(v\mymid z)\rd v \\
 &\leq \int_s^\infty 
\frac{v^{(p+n)/2}}{\psi(z;0,1)}
\int_0^1 \frac{t^{p/2+a}(1-t)^b}{(1-zt)^{p/2+a+b+2} }
 \exp\left(-\frac{v}{4(1-zt)}\right) \rd t\rd v \\
 &\leq \frac{\psi(z,0,2)}{\psi(z;0,1)} \leq \frac{T_2(0,2)}{T_1(0,1)}\coloneqq C_3(0),
\end{align*}
where the third inequality follows from Lemma \ref{lem:psi}.

 For $k>0$, note by Part \ref{lem:log.poly.2},
\begin{equation}\label{Q9Q9.1}
 |\log v|^k \leq
\left(\frac{2k}{n/2-a}\right)^k\left(v^{(n/2-a)/2}+ v^{-(n/2-a)/2}\right).
\end{equation}
Then  we have
\begin{align*}
& \exp(s/4)\int_s^\infty  |\log v|^k f(v\mymid z)\rd v \\
 &\leq 
 \int_s^\infty \left(\frac{2k}{n/2-a}\right)^k\left(v^{(n/2-a)/2}+ v^{-(n/2-a)/2}\right) \\
&\quad\times\frac{v^{(p+n)/2}}{\psi(z;0,1)}
\int_0^1 \frac{t^{p/2+a}(1-t)^b}{(1-zt)^{p/2+a+b+2} }
 \exp\left(-\frac{v}{4(1-zt)}\right) \rd t\rd v \\
 &\leq 
 \left(\frac{2k}{n/2-a}\right)^k\frac{\psi(z,-1/2,2)+\psi(z,1/2,2)}{\psi(z;0,1)} \\
 &\leq \left(\frac{2k}{n/2-a}\right)^k \frac{T_2(-1/2,2)+T_2(1/2,2)}{T_1(0,1)} \\
 &\coloneqq C_3(k)
\end{align*} 
where the third inequality follows from Lemma \ref{lem:psi}.
This completes the proof.
    
\subsubsection{Proof of Lemma \ref{lem.2.8.3}}
\mbox{}

[Part \ref{lem.2.8.3.1}] \ 
Assume $s<1$ equivalently $\log(1/s)>0$.
%Fix $z\in(0,1)$ and $i\in\mathbb{Z}$ arbitrarily.
Then by Lemma \ref{Hiineq.1}, we have
\begin{align*}
&\{i+\log(1/s)\} E\left[H_i(V/s)\mymid z\right]  \\
 &=\int_0^s \frac{i+\log(1/s)}{i+\log (s/v)}f(v\mymid z)\rd v
 +\int_s^\infty \frac{i+\log(1/s)}{i+\log (v/s)}f(v\mymid z)\rd v \\
&\geq \int_s^\infty \frac{i+\log(1/s)}{i+\log (v/s)}f(v\mymid z)\rd v \\
 &\geq \int_s^\infty
 \left(1-\frac{\log v}{i+\log(1/s)}-\frac{|\log v|^3}{\{i+\log(1/s)\}^2}\right)f(v\mymid z)\rd v \\
 &\geq %\left\{
 1-\int_0^s f(v\mymid z)\rd v -
 \frac{E\left[\log V\mymid z\right]}{i+\log(1/s)} %\right. \\
 %&\quad \left.
 -\frac{\int_0^s |\log v| f(v\mymid z)\rd v}{i+\log(1/s)}
 -\frac{E\left[|\log V|^3\mymid z\right]}{\{i+\log(1/s)\}^2}
 %\right\}
 .
\end{align*}
%\begin{align*}
% E\left[H_i(V/s)\mymid z\right] 
% &=\int_0^s \frac{f(v\mymid z)}{i+\log (s/v)}\rd v
% +\int_s^\infty \frac{f(v\mymid z)}{i+\log (v/s)}\rd v \\
%&\geq \int_s^\infty \frac{f(v\mymid z)}{i+\log (v/s)}\rd v \\
% &\geq \frac{1}{i+\log(1/s)}\int_s^\infty
% \left(1-\frac{\log v}{i+\log(1/s)}-\frac{|\log v|^3}{\{i+\log(1/s)\}^2}\right)f(v\mymid z)\rd v \\
% &\geq \frac{1}{i+\log(1/s)}\left\{1-\int_0^s f(v\mymid z)\rd v -
% \frac{E\left[\log V\mymid z\right]}{i+\log(1/s)} \right. \\
%&\quad \left. -\frac{\int_0^s |\log v| f(v\mymid z)\rd v}{i+\log(1/s)}
% -\frac{E\left[|\log V|^3\mymid z\right]}{\{i+\log(1/s)\}^2} \right\}.
%\end{align*}
By Lemma \ref{lem.2.8.2}, there exists $T_4>0$ such that 
\begin{align*}
\int_0^s f(v\mymid z)\rd v +\frac{\int_0^s |\log v| f(v\mymid z)\rd v}{i+\log(1/s)}+\frac{E\left[|\log V|^3\mymid z\right]}{\{i+\log(1/s)\}^2}\leq \frac{T_4}{\{1+\log(1/s)\}^2}
\end{align*}
for all $s\in(0,1)$ and hence
\begin{equation}\label{SMALL.S.1}
\{i+\log(1/s)\} E\left[H_i(V/s)\mymid z\right]\geq
g(s,z;i)
\end{equation}
where
\begin{equation}\label{korekore.0}
 g(s,z;i)= 1- \frac{E\left[\log V\mymid z\right]}{i+\log(1/s)}-\frac{T_4}{\{1+\log(1/s)\}^2}.
\end{equation}

Further
\begin{equation}\label{korekore}
 \frac{|E\left[\log V\mymid z\right]|}{i+\log(1/s)}< \frac{E\left[|\log V| \mymid z\right]}{1+\log(1/s)}\leq \frac{C_4(1)}{1+\log(1/s)}
\end{equation}
and hence 
\begin{align*}
  g(s,z;i)\geq 0,
\end{align*}
for all $s< C_5$ where
\begin{equation}
 C_5=1/\exp(C_4(1)+T_4).
\end{equation}
Consider $\{g(s,z;i)\}^2$ for all $s< C_5$. 
Then, by \eqref{korekore.0} and \eqref{korekore}, 
\begin{align*}
 \{g(s,z;i)\}^2
&\geq 1 -2\frac{E\left[\log V\mymid z\right]}{i+\log(1/s)}+\frac{2T_4 E\left[\log V\mymid z\right]}{\{1+\log(1/s)\}^2\{i+\log(1/s)\}}
 -\frac{2T_4}{\{1+\log(1/s)\}^2} \\
 &\geq %g_2(s,z;i)
 1-2\frac{E\left[\log V\mymid z\right]}{i+\log(1/s)}-\frac{C_6}{\{1+\log(1/s)\}^2},
\end{align*}
where
%\begin{equation}
% g_2(s,z;i)=
%
%\end{equation}
\begin{align*}
 C_6= 2T_4\{C_4(1)+1\}.
\end{align*}
This completes the proof for Part \ref{lem.2.8.3.1}.
%Again by \eqref{korekore}, $g_2(s,z;i) \geq 0$ for all $s<\tau_2$ where
%\begin{equation}\label{eq:tau_2}
% \tau_2=1/\exp(2Q_9(1)+2T_4+2Q_9(1)T_4).
%\end{equation}
%Clearly $\tau_2<\tau_1<1$ and let
%\begin{align*}
% \tau_2=C_5.
%\end{align*}
%Then for all $0<s<C_5$,
%\begin{equation}\label{omake.-1}
%\{E\left[H_i(V/s)\mymid z\right]\}^2\geq \{g_1(s,z;i)\}^2\geq g_2(s,z;i) \geq 0.
%\end{equation}
%
\smallskip

[Part \ref{lem.2.8.3.2}] \
Assume $s<1$ equivalently $\log(1/s)>0$.
We consider $ E\left[H_i^2(V/s)\mymid z\right] $ given by
\begin{equation}\label{omake.0}
E\left[H_i^2(V/s)\mymid z\right] 
=\int_0^s \frac{f(v\mymid z)}{\{i+\log (s/v)\}^2}\rd v
 +\int_s^\infty \frac{f(v\mymid z)}{\{i+\log (v/s)\}^2}\rd v .
\end{equation}
Note
\begin{equation}\label{omake.1}
 \begin{split}
 \int_0^s \frac{f(v\mymid z)}{\{i+\log (s/v)\}^2}\rd v 
& \leq\frac{1}{i^2}\int_0^s f(v\mymid z)\rd v \\
&\leq\frac{\left\{1+\log(1/s)\right\}^2}{\{i+\log(1/s)\}^2}\int_0^s f(v\mymid z)\rd v \\
&=\frac{\left\{1+\log(1/s)\right\}^4\int_0^s f(v\mymid z)\rd v }{\{i+\log(1/s)\}^2\left\{1+\log(1/s)\right\}^2}.
\end{split}
\end{equation}
In the numerator above, by Lemma \ref{lem.2.8.2}, there exists $T_5>0$ such that
\begin{equation}\label{omake.2}
 \left\{1+\log(1/s)\right\}^4\int_0^s f(v\mymid z)\rd v\leq T_5
\end{equation}
for all $s\in(0,1)$.
Further, by Lemma \ref{Hiineq.1}, we have
\begin{equation}\label{omake.3}
\begin{split}
&\{i+\log (1/s)\}^2\int_s^\infty \frac{f(v\mymid z)}{\{i+\log (v/s)\}^2}\rd v \\
 &\leq \int_s^\infty
 \left(1-\frac{2\log v}{i+\log(1/s)}+4\frac{\sum_{j=2}^6|\log v|^j}{\{i+\log(1/s)\}^2}\right)f(v\mymid z)\rd v \\
& \leq\int_0^\infty
 \left(1-\frac{2\log v}{i+\log(1/s)}+4\frac{\sum_{j=2}^6|\log v|^j}{\{i+\log(1/s)\}^2}\right)f(v\mymid z)\rd v \\
 &=1-2\frac{E[\log V\mymid z]}{i+\log(1/s)} +4\frac{\sum_{j=2}^6 C_4(j)
 %\sum_{j=2}^6E[|\log V|^j\mymid z]
 }{\{i+\log(1/s)\}^2}.
\end{split}
\end{equation}
By \eqref{omake.0} -- \eqref{omake.3} and Lemma \ref{lem.2.8.2},
we have
\begin{equation}\label{omake.4}
{\{i+\log(1/s)\}^2} E\left[H_i^2(V/s)\mymid z\right] \leq 
1-2\frac{E[\log V\mymid z]}{i+\log(1/s)} +\frac{C_7}{\{1+\log(1/s)\}^2},
\end{equation}
for all $s\in(0,1)$, where
\begin{align*}
 C_7=T_5+4\sum_{j=2}^6 C_4(j).
\end{align*}
This completes the proof for Part \ref{lem.2.8.3.2}.

\smallskip

%Note, for all $s<\tau_2$, 
%\begin{equation}\label{omake.5}
% \begin{split}
%  1-2\frac{E[\log V\mymid z]}{i+\log(1/s)} &\geq 1-2\frac{Q_9(1)}{\log(1/s)}\\
% &  \geq 1-\frac{Q_9(1)}{Q_9(1)+T_4Q_9(1)+T_4} \\
% &=\frac{T_4Q_9(1)+T_4}{Q_9(1)+T_4Q_9(1)+T_4} \\
% &> \frac{T_4}{T_4+1}.
%\end{split}
%\end{equation}
%Finally, for all $z\in(0,1)$, all $i\in\mathbb{Z}$ and all $s<\tau_2$ given by \eqref{eq:tau_2}, we have 
%\begin{equation}\label{B.important.1}
% \begin{split}
% & 1-\frac{\{E\left[H_i(V/s)\mymid z\right]\}^2}{E\left[H_i^2(V/s)\mymid z\right]} \\
% & \leq
% 1-\frac{1-2\frac{E\left[\log V\mymid z\right]}{i+\log(1/s)}-2\frac{T_4\{Q_9(1)+1\}}{\{1+\log(1/s)\}^2}}{1-2\frac{E[\log V\mymid z]}{i+\log(1/s)} +\frac{Q_{13}}{\{1+\log(1/s)\}^2}} \\
% &=\frac{2T_4(Q_9(1)+1)+Q_{13}}{\{1+\log(1/s)\}^2}\frac{1}{1-2\frac{E[\log V\mymid z]}{i+\log(1/s)} +\frac{Q_{13}}{\{1+\log(1/s)\}^2}} \\
% &\leq \frac{\{2T_4(Q_9(1)+1)+Q_{13}\}(T_4+1)}{T_4}
% \frac{1}{\{1+\log(1/s)\}^2}, \\
% &=  \frac{Q_{14}}{\{1+\log(1/s)\}^2},
% \end{split}
%\end{equation}
% where
%\begin{align*}
% Q_{14}=\{2(Q_9(1)+1)+Q_{13}/T_4\}(T_4+1)
%\end{align*}
%and the first inequality follows from \eqref{omake.-1} and \eqref{omake.4} and the second inequality follows from \eqref{omake.5}.
%
 
% \subsubsection{Case: $s>1$}
%\label{sec.large.s}

[Part \ref{lem.2.8.3.3}] \ 
Assume $s>1$ equivalently $\log s>0$.
Then by Lemma \ref{Hiineq.2}, we have
\begin{align*}
&\{i+\log s\} E\left[H_i(V/s)\mymid z\right] \\
 &=\int_0^s \frac{i+\log s}{i+\log (s/v)}f(v\mymid z)\rd v
 +\int_s^\infty \frac{i+\log s}{i+\log (v/s)}f(v\mymid z)\rd v \\
&\geq \int_0^s \frac{i+\log s}{i+\log (s/v)}f(v\mymid z)\rd v \\
 &\geq \int_0^s
 \left(1+\frac{\log v}{i+\log s }-\frac{|\log v|^3}{(i+\log s )^2}\right)f(v\mymid z)\rd v \\
 &\geq 1-\int_s^\infty f(v\mymid z)\rd v +\frac{E\left[\log V\mymid z\right]}{i+\log s }% \right. \\
 %&\quad \left.
 -\frac{\int_s^\infty |\log v| f(v\mymid z)\rd v}{i+\log s }
 -\frac{E\left[|\log V|^3\mymid z\right]}{(i+\log s )^2}.
% \right\}.
\end{align*}
By Lemma \ref{lem.2.8.2}, there exists $T_6>0$ such that 
\begin{align*}
\int_s^\infty f(v\mymid z)\rd v +\frac{\int_s^\infty |\log v| f(v\mymid z)\rd v}{i+\log s }+\frac{E\left[|\log V|^3\mymid z\right]}{(i+\log s )^2}\leq \frac{T_6}{(1+\log s)^2},
\end{align*}
for all $s\in(1,\infty)$ and hence
\begin{equation}\label{LARGE.S.1}
(i+\log s) E\left[H_i(V/s)\mymid z\right]\geq
g(s,z;i)
\end{equation}
where
\begin{equation}
 g(s,z;i)
  = 1+ \frac{E\left[\log V\mymid z\right]}{i+\log s }
			-\frac{T_6}{(1+\log s)^2}.
\end{equation}
Further, by \eqref{korekore}, we have
\begin{align*}
  g(s,z;i)\geq 0,
\end{align*}
for all $s> C_8$ where
\begin{equation}\label{tau3}
 C_8=\exp(C_4(1)+T_6).
\end{equation}

Consider $\{g(s,z;i)\}^2$ for all $s> C_8$. 
Then, by \eqref{korekore.0} and \eqref{korekore}, 
\begin{align*}
 \{g(s,z;i)\}^2
&\geq 1 +2\frac{E\left[\log V\mymid z\right]}{i+\log s} - \frac{2T_6 E\left[\log V\mymid z\right]}{(1+\log s)^2(i+\log s)}
 -\frac{2T_6}{(1+\log s)^2} \\
 &\geq %g_2(s,z;i)
 1+2\frac{E\left[\log V\mymid z\right]}{i+\log s }-\frac{C_9}{(1+\log s)^2},
\end{align*}
where
%\begin{equation}
% g_2(s,z;i)=
%
%\end{equation}
\begin{align*}
 C_9= 2T_6\{C_4(1)+1\}.
\end{align*}
This completes the proof for Part \ref{lem.2.8.3.3}.

%Consider $\{g_3(s,z;i)\}^2$ for all $s> \tau_3$. 
%Then, 
%\begin{align*}
%  \{g_3(s,z;i)\}^2  \geq g_4(s,z;i)
%\end{align*}
%where
%\begin{equation}
% g_4(s,z;i)=
%  \frac{1}{(i+\log s)^2}
%  \left(1+2\frac{E\left[\log V\mymid z\right]}{i+\log s }-
%   2\frac{Q_{15}(Q_9(1)+1)}{(1+\log s)^2}\right).
%\end{equation}
%

[Part \ref{lem.2.8.3.4}] \ 
Assume $s>1$ equivalently $\log s >0$.
We consider $ E\left[H_i^2(V/s)\mymid z\right] $ given by
%\begin{equation}\label{omake.0}
%E\left[H_i^2(V/s)\mymid z\right] 
%=\int_0^s \frac{f(v\mymid z)}{\{i+\log (s/v)\}^2}\rd v
% +\int_s^\infty \frac{f(v\mymid z)}{\{i+\log (v/s)\}^2}\rd v .
%\end{equation}
\begin{equation}\label{l.omake.0}
E\left[H_i^2(V/s)\mymid z\right] 
=\int_0^s \frac{f(v\mymid z)}{\{i+\log (s/v)\}^2}\rd v
 +\int_s^\infty \frac{f(v\mymid z)}{\{i+\log (v/s)\}^2}\rd v .
\end{equation}
Note
\begin{equation}\label{l.omake.1}
 \begin{split}
 \int_s^\infty \frac{f(v\mymid z)}{\{i+\log (v/s)\}^2}\rd v 
& \leq\frac{1}{i^2}\int_s^\infty f(v\mymid z)\rd v \\
&\leq\frac{(1+\log s)^2}{(i+\log s)^2}\int_s^\infty f(v\mymid z)\rd v \\
&=\frac{(1+\log s)^4\int_s^\infty f(v\mymid z)\rd v }{(1+\log s)^2 (i+\log s)^2}.
\end{split}
\end{equation}
%\begin{equation}\label{omake.1}
% \begin{split}
% \int_0^s \frac{f(v\mymid z)}{\{i+\log (s/v)\}^2}\rd v 
%& \leq\frac{1}{i^2}\int_0^s f(v\mymid z)\rd v \\
%&\leq\frac{\left\{1+\log(1/s)\right\}^2}{\{i+\log(1/s)\}^2}\int_0^s f(v\mymid z)\rd v \\
%&=\frac{\left\{1+\log(1/s)\right\}^4\int_0^s f(v\mymid z)\rd v }{\{i+\log(1/s)\}^2\left\{1+\log(1/s)\right\}^2}.
%\end{split}
%\end{equation}
In the numerator above, by Lemma \ref{lem.2.8.2}, there exists $T_7>0$ such that
\begin{equation}\label{l.omake.2}
 (1+\log s)^4\int_s^\infty f(v\mymid z)\rd v
%  \left\{1+\log(1/s)\right\}^4\int_0^s f(v\mymid z)\rd v
  \leq T_7
\end{equation}
for all $s\in(1,\infty)$.
Further, by Lemma \ref{Hiineq.2}, we have
\begin{equation}\label{l.omake.3}
\begin{split}
& (i+\log s)^2\int_0^s
 \left(1+\frac{2\log v}{i+\log s }+4\frac{\sum_{j=2}^6|\log v|^j}{(i+\log s)^2}\right)f(v\mymid z)\rd v \\
& \leq\int_0^\infty
 \left(1+\frac{2\log v}{i+\log s }+4\frac{\sum_{j=2}^6|\log v|^j}{(i+\log s)^2}\right)f(v\mymid z)\rd v \\
&=1+2\frac{E[\log V\mymid z]}{i+\log s } +4\frac{\sum_{j=2}^6E[|\log V|^j\mymid z]}{(i+\log s)^2}.
\end{split}
\end{equation}

%\begin{equation}\label{omake.3}
%\begin{split}
%&\{\}^2 \int_0^s \frac{f(v\mymid z)}{\{i+\log (v/s)\}^2}\rd v \\
% &\leq \int_s^\infty
% \left(1-\frac{2\log v}{i+\log(1/s)}+4\frac{\sum_{j=2}^6|\log v|^j}{\{i+\log(1/s)\}^2}\right)f(v\mymid z)\rd v \\
%& \leq\int_0^\infty
% \left(1-\frac{2\log v}{i+\log(1/s)}+4\frac{\sum_{j=2}^6|\log v|^j}{\{i+\log(1/s)\}^2}\right)f(v\mymid z)\rd v \\
% &=1-2\frac{E[\log V\mymid z]}{i+\log(1/s)} +4\frac{\sum_{j=2}^6 C_4(j)
% %\sum_{j=2}^6E[|\log V|^j\mymid z]
% }{\{i+\log(1/s)\}^2}.
%\end{split}
%\end{equation}
By \eqref{l.omake.0} -- \eqref{l.omake.3} and Lemma \ref{lem.2.8.2},
we have
\begin{equation}\label{l.omake.4}
(i+\log s)^2 E\left[H_i^2(V/s)\mymid z\right] \leq 
1+2\frac{E[\log V\mymid z]}{i+\log s } +\frac{C_{10}}{(1+\log s)^2}
\end{equation}
%\begin{equation}\label{omake.4}
%{\{i+\log(1/s)\}^2} E\left[H_i^2(V/s)\mymid z\right] \leq 
%1-2\frac{E[\log V\mymid z]}{i+\log(1/s)} +\frac{C_7}{\{1+\log(1/s)\}^2},
%\end{equation}
for all $s\in(1,\infty)$, where
\begin{align*}
 C_{10}=T_7+4\sum_{j=2}^6 C_4(j).
\end{align*}
This completes the proof for Part \ref{lem.2.8.3.4}.

 \section{Preliminary results for lemmas in Appendix A}
% \section{Preliminary results for lemmas in Appendix \ref{app.lemmas.sec}}

  \subsection{Properties of $\pi(r\mymid a,b)$}
\label{sec:pi}
\begin{lemma}\label{lem:pirab}
 \begin{enumerate}
  \item \label{lem:pirab.1}There exists $Q_1$ such that
\begin{align*}
 r\frac{|\pi'(r\mymid a,b)|}{\pi(r\mymid a,b)}<Q_1
\end{align*}
for all $r\geq 0$.
\item \label{lem:pirab.2}Assume $b>-1/2$. There exists $Q_3$ such that
\begin{align*}
 \pi(r\mymid a,b)\{r^{-1/2}I_{[0,1]}(r)+r^{-1}I_{(1,\infty)}(r)\}<Q_3
 \pi(r\mymid a+1,\tilde{b})
\end{align*}
where $\dps \tilde{b}=\begin{cases}-1/2 & b\geq 0\\
	  b-1/2 & -1/2<b<0.
	 \end{cases}$
 \end{enumerate}
\end{lemma}
\begin{proof}
By the change of variables $\xi=\lambda/(1-\lambda)$, $\pi(r\mymid a,b)$ is given by
 \begin{equation*}
  \pi(r\mymid a,b)=
   \int_0^\infty\frac{1}{(2\pi\xi)^{p/2}}\exp\left(-\frac{r}{2\xi}\right)
  \frac{\xi^b(1+\xi)^{-a-b-2}}{B(a+1,b+1)}\rd \xi.
 \end{equation*}
Note that 
\begin{equation}\label{limxi.1}
\lim_{\xi\to 0}\frac{\xi^{-p/2} \xi^b(1+\xi)^{-a-b-2} }
{\xi^{-p/2+b}}=1,
\end{equation}
and
\begin{equation}\label{limxi.2}
 \lim_{\xi\to \infty}\frac{\xi^{-p/2} \xi^b(1+\xi)^{-a-b-2} }
{\xi^{-p/2-a-2}}=1.
\end{equation}
In the following, $f(r)\approx g(r)$ stands for $ \lim f(r)/g(r)=1$ as $r\to 0$ or $r\to\infty$.
By a standard Tauberian Theorem with \eqref{limxi.1}, we have
\begin{equation}\label{for.part2.lem:pirab.1}
\{(2\pi)^{p/2}B(a+1,b+1)\}\pi(r\mymid a,b)\approx \Gamma(p/2+a+1)(2/r)^{p/2+a+1} %\text{ as }r\to\infty.
\end{equation}
as $r\to\infty$.
 Further, by the Tauberian Theorem with \eqref{limxi.2}, we have, as $r\to 0$, 
 \begin{equation}\label{for.part2.lem:pirab.2}
\{(2\pi)^{p/2}B(a+1,b+1)\} \pi(r\mymid a,b)\approx \Gamma(p/2-b-1)(2/r)^{p/2-b-1}%\text{ as }r\to 0\text{ for }, 
 \end{equation}
for $ -1<b<p/2-1$;
\begin{equation}\label{for.part2.lem:pirab.3}
 \{(2\pi)^{p/2}B(a+1,b+1)\} \pi(r\mymid a,b)\approx \log (1/r)
\end{equation}
for $b=p/2-1$; and
\begin{equation}\label{for.part2.lem:pirab.4}
\{(2\pi)^{p/2}B(a+1,b+1)\} \pi(0\mymid a,b) = \frac{B(p/2+a+1,b-p/2+1)}{B(a+1,b+1)}%\text{ for }b>p/2-1. 
\end{equation}
for $b>p/2-1$. The derivative of $\pi(r\mymid a,b)$ is
 \begin{equation*}
-2  \pi'(r\mymid a,b)=
   \int_0^\infty\frac{1}{(2\pi)^{p/2}\xi^{p/2+1}}\exp\left(-\frac{r}{2\xi}\right)
  \frac{\xi^b(1+\xi)^{-a-b-2}}{B(a+1,b+1)}\rd \xi.
 \end{equation*}
Similarly, we have
\begin{align*}
\{-2(2\pi)^{p/2}B(a+1,b+1)\}\pi'(r\mymid a,b)\approx \Gamma(p/2+a+2)(2/r)^{p/2+a+2}% \text{ as }r\to\infty.
\end{align*}
as $r\to\infty$. Further, as $r\to 0$, 
\begin{align*}
\{-2(2\pi)^{p/2}B(a+1,b+1)\} \pi'(r\mymid a,b)\approx \Gamma(p/2-b)(2/r)^{p/2-b}%\text{ as }r\to 0\text{ for }, 
\end{align*}
for $ -1<b<p/2$; 
\begin{align*}
 \{-2(2\pi)^{p/2}B(a+1,b+1)\} \pi'(r\mymid a,b) \approx \log(1/r)
\end{align*}
for $b=p/2$; and
\begin{align*}
\{-2(2\pi)^{p/2}B(a+1,b+1)\} \pi'(r\mymid a,b)|_{r=0} = \frac{B(p/2+a+2,b-p/2)}{B(a+1,b+1)}%\text{ for }b>p/2. 
\end{align*}
for $b>p/2$. Hence
\begin{equation}\label{eq.bound}
 \begin{split}
\lim_{r\to\infty} r\frac{\pi'(r\mymid a,b)}{\pi(r\mymid a,b)}&=-(p/2+a+1) \\
 \lim_{r\to 0} r\frac{\pi'(r\mymid a,b)}{\pi(r\mymid a,b)}&=
 \begin{cases}
  -(p/2-1-b) & -1<b<p/2-1 \\
  0 & b\geq p/2-1,
 \end{cases}
\end{split}
\end{equation}
which completes the proof of Part \ref{lem:pirab.1}.

Part \ref{lem:pirab.2} follows from \eqref{for.part2.lem:pirab.1}--\eqref{for.part2.lem:pirab.4}.
\end{proof}

\begin{lemma}\label{lem.ASMS}
 Let
 \begin{align*}
   k(r)=r^{1/2}I_{[0,1]}(r)+I_{(1,\infty)}(r).
 \end{align*}
Then we have
 \begin{align*}
& m\left(\frac{\pi(\theta,\eta) \eta}{k(\eta\|\theta\|^2)} \left\{1 - \frac{\sqrt{m(\pi \eta)}}{\sqrt{m(\pi h_i^2\eta)}}h_i\right\}^2\right) \\ &\leq Q_2
m\left(\pi(\theta,\eta) \eta \left\{1 - \frac{\sqrt{m(\pi \eta)}}{\sqrt{m(\pi h_i^2\eta)}}h_i\right\}^2\right)
 \end{align*}
 where
\begin{align*}
Q_2 = \frac{\int_0^1 \pi(r\mymid a,b)r^{p/2-3/2}\rd r}{\int_0^1 \pi(r\mymid a,b)r^{p/2-1}\rd r}.
\end{align*}
\end{lemma}
\begin{proof}
 Note
 \begin{align}
% \begin{split}
& m\left(\frac{\pi(\theta,\eta) \eta}{k(\eta\|\theta\|^2)} \left\{1 - \frac{\sqrt{m(\pi \eta)}}{\sqrt{m(\pi h_i^2\eta)}}h_i\right\}^2\right)\label{ASMS.1}\\
& = \int_0^\infty \int_{\mbR^p}
\frac{\pi(\theta,\eta) \eta}{k(\eta\|\theta\|^2)} 
  \left\{1 - \frac{\sqrt{m(\pi \eta)}}{\sqrt{m(\pi h_i^2\eta)}}h_i\right\}^2
f_{x }(x \mymid\theta,\eta) f_s(s\mymid\eta)\rd \theta\rd \eta \notag\\
& = \int_0^\infty \left(\int_{\mbR^p}
\frac{\pi(\theta,\eta) \eta}{k(\eta\|\theta\|^2)} f_{x }(x \mymid\theta,\eta) \rd \theta\right)
  \left\{1 - \frac{\sqrt{m(\pi \eta)}}{\sqrt{m(\pi h_i^2\eta)}}h_i\right\}^2
f_s(s\mymid\eta)\rd \eta .\notag
%\end{split} 
 \end{align}
The integrand in parentheses can be decomposed as
% \begin{align*}
%  \left\{1 - \frac{\sqrt{m(\pi \eta)}}{\sqrt{m(\pi h_i^2\eta)}}h_i\right\}^2
%f_s(s\mymid\eta)
% \end{align*}
%does not include $\theta$.
% Further
\begin{equation}\label{ASMS.2}
 \begin{split}
& \int_{\mbR^p}
 \frac{\pi(\theta,\eta) \eta}{k(\eta\|\theta\|^2)}
 %\left\{1 - \frac{\sqrt{m(\pi \eta)}}{\sqrt{m(\pi h_i^2\eta)}}h_i\right\}^2
 f_{x }(x \mymid\theta,\eta)
 %f_s(s\mymid\eta)
   \rd \theta \\
 &=\int_{\eta\|\theta\|^2\leq 1}
 \frac{\pi(\theta,\eta) \eta}{\eta^{1/2}\|\theta\|}f_{x }(x \mymid\theta,\eta)\rd \theta 
+ \int_{\eta\|\theta\|^2>1}
 \pi(\theta,\eta) \eta f_{x }(x \mymid\theta,\eta)\rd \theta. 
\end{split}
\end{equation}
% Note also
%\begin{align*}
%& \int_{\eta\|\theta\|^2>1}
% \frac{\pi(\theta,\eta) \eta}{k(\eta\|\theta\|^2)}f_{x }(x \mymid\theta,\eta)\rd \theta \\
%&=\int_{\eta\|\theta\|^2>1}
% \pi(\theta,\eta) \eta f_{x }(x \mymid\theta,\eta)\rd \theta. 
%\end{align*}
In the first term of \eqref{ASMS.2}, we have 
\begin{equation}\label{ASMS.3}
 \begin{split}
&\int_{\eta\|\theta\|^2\leq 1}
 \frac{\pi(\theta,\eta) \eta}{\eta^{1/2}\|\theta\|}f_{x }(x \mymid\theta,\eta)\rd \theta \\
&=\int_{\eta\|\theta\|^2\leq 1}
 \frac{\eta^{p/2}\pi(\eta\|\theta\|^2\mymid a,b)}{\eta^{1/2}\|\theta\|}f_{x }(x \mymid\theta,\eta)\rd \theta \\
&=\eta^{p/2}\int_{\|\mu\|^2\leq 1}
 \frac{\pi(\|\mu\|^2\mymid a,b)}{\|\mu\|}\frac{1}{(2\pi)^{p/2}}
 \exp\left(-\frac{\|\mu-\eta^{1/2}x\|^2}{2}\right)\rd \mu .
\end{split}
\end{equation}
Note $\| \mu\|^2$ may be regarded as a non-central chi-square random variable
 with $p$ degrees of freedom and $\eta\|  x \|^2$ non-centrality parameter.
 Let
 \begin{align*}
 a_j( \eta\|x\|^2 )=\frac{1}{\Gamma(p/2+j)2^{p/2+j}}\frac{(\eta\|  z \|^2/2)^j}{j!}\exp(-\eta\|  z \|^2/2).
 \end{align*}
Then we have
\begin{equation}\label{ASMS.4}
 \begin{split}
& \int_{\|\mu\|^2\leq 1}
 \frac{\pi(\|\mu\|^2\mymid a,b)}{\|\mu\|}\frac{1}{(2\pi)^{p/2}}
 \exp\left(-\frac{\|\mu-\eta^{1/2}x\|^2}{2}\right)\rd \mu  \\
&=\sum_{j=0}^\infty a_j(\eta\|x\|^2)\int_0^1 \pi(r\mymid a,b)r^{p/2+j-1-1/2}\rd r.
\end{split}
\end{equation}
By the correlation inequality, we have
\begin{equation}\label{ASMS.5}
 \begin{split}
 & \int_0^1 \pi(r\mymid a,b)r^{p/2+j-1-1/2}\rd r \\
 &\leq 
  \frac{\int_0^1 \pi(r\mymid a,b)r^{p/2-3/2}\rd r}{\int_0^1 \pi(r\mymid a,b)r^{p/2-1}\rd r}
  \int_0^1 \pi(r\mymid a,b)r^{p/2-1+j}\rd r \\
 &= Q_2
  \int_0^1 \pi(r\mymid a,b)r^{p/2-1+j}\rd r 
\end{split}
\end{equation}
 for $j\geq 0$, where $ Q_2$ is greater than $1$ by definition.
By \eqref{ASMS.5},
\begin{equation}\label{ASMS.6}
 \int_{\eta\|\theta\|^2\leq 1}
 \frac{\pi(\theta,\eta) \eta}{\eta^{1/2}\|\theta\|}f_{x }(x \mymid\theta,\eta)\rd \theta
\leq Q_2 \int_{\eta\|\theta\|^2\leq 1}
 \pi(\theta,\eta) \eta f_{x }(x \mymid\theta,\eta)\rd \theta.
\end{equation} 
 Also since $Q_2>1$ and by \eqref{ASMS.2}, we have
\begin{equation}\label{ASMS.7}
 \int_{\eta\|\theta\|^2>1}
 \pi(\theta,\eta) \eta f_{x }(x \mymid\theta,\eta)\rd \theta
\leq Q_2\int_{\eta\|\theta\|^2>1}
 \pi(\theta,\eta) \eta f_{x }(x \mymid\theta,\eta)\rd \theta.
\end{equation}
Hence the result follows from \eqref{ASMS.1}, \eqref{ASMS.2}, \eqref{ASMS.6} and \eqref{ASMS.7}.
\end{proof}

%\begin{lemma}
% There exists $C$ such that
%\begin{align*}
% \frac{\pi(r)}{r}\leq 
% \frac{C}{(2\pi)^{p/2}}\int_0^1\left(\frac{\lambda}{1-\lambda}\right)^{p/2}
%\exp\left(-\frac{\lambda}{1-\lambda}\frac{r}{2}\right)
%\frac{\lambda^{a'}(1-\lambda)^{b'}}{B(a'+1,b'+1)} \rd\lambda,
%\end{align*}
%where $a'=a+1$
%\begin{align*}
% b'=\begin{cases}b-1 & b<p/2-1 \\
%     p/2-2 & b>p/2-1
%    \end{cases}
%\end{align*}
%\end{lemma}

\begin{lemma}\label{lem:polyexp}
 Let $c_1>c_2>0$. Then
\begin{align*}
 \max_{y\in(0,\infty)}y^{-c_1}\exp(-c_2/y)< (c_1/c_2)^{c_1}.
\end{align*} 
\end{lemma}
\begin{proof}
 Straightforward.
\end{proof}

\begin{lemma}\label{Hiineq.1}
For $s<1$ and $v\geq s$,
\begin{align*}
%& \frac{1}{i+\log v+\log(1/s)} \\
%& \geq
%  \frac{1}{i+\log(1/s)}\left(1-\frac{\log v}{i+\log(1/s)}- \frac{|\log v|^3}{\{i+\log(1/s)\}^2}\right), \\
%& \frac{1}{i+\log v+\log(1/s)} \\
%& \leq
%  \frac{1}{i+\log(1/s)}\left(1-\frac{\log v}{i+\log(1/s)}+ \frac{|\log v|^2+|\log v|^3}{\{i+\log(1/s)\}^2}\right)
& \frac{i+\log(1/s)}{i+\log v+\log(1/s)} 
 \geq
  1-\frac{\log v}{i+\log(1/s)}- \frac{|\log v|^3}{\{i+\log(1/s)\}^2}, \\
& \frac{i+\log(1/s)}{i+\log v+\log(1/s)} \leq
  1-\frac{\log v}{i+\log(1/s)}+ \frac{|\log v|^2+|\log v|^3}{\{i+\log(1/s)\}^2}
\end{align*}
and
 \begin{align*}
\left(\frac{i+\log(1/s)}{i+\log v+\log(1/s)}\right)^2 
\leq
1-2\frac{\log v}{i+\log(1/s)}+ 4\frac{\sum_{i=2}^6|\log v|^i}{\{i+\log(1/s)\}^2}.
\end{align*}
\end{lemma}
\begin{proof}
For $v\geq s$,
\begin{align*}
& \frac{i+\log(1/s)}{i+\log v+\log(1/s)} \\
 &= 1-\frac{\log v}{i+\log v+\log(1/s)} \\
%& = \frac{1}{i+\log(1/s)}\left(1-\frac{\log v}{i+\log(1/s)}+ \frac{\{\log v\}^2}{i+\log v+\log(1/s)}\right) \\
& = 1-\frac{\log v}{i+\log(1/s)}+ \frac{\{\log v\}^2}{\{i+\log(1/s)\}\{i+\log v+\log(1/s)\}} \\
 & = 1-\frac{\log v}{i+\log(1/s)}+ \frac{\{\log v\}^2}{\{i+\log(1/s)\}^2}
- \frac{\{\log v\}^3}{\{i+\log(1/s)\}^2\{i+\log v+\log(1/s)\}}.
\end{align*}
Then three inequalities follow from the inequality
\begin{align*}
\left| \frac{\{\log v\}^3}{\{i+\log(1/s)\}^2\{i+\log v+\log(1/s)\}}\right|
 \leq \frac{|\log v|^3}{\{i+\log(1/s)\}^2}
\end{align*}
for $i\geq 1$. 
\end{proof}
In the similar way, we have a following result.
%For $v\leq s$,
%\begin{align*}
%& \frac{1}{i+\log s+\log(1/v)} \\
% &= \frac{1}{i+\log s}\left(1-\frac{\log (1/v)}{i+\log 1/v+\log s}\right) \\
%%& = \frac{1}{i+\log s}\left(1-\frac{\log (1/v)}{i+\log s}+ \frac{\{\log (1/v)\}^2}{(i+\log s)(i+\log (1/v)+\log s)}\right) \\
%& = \frac{1}{i+\log s}\left(1-\frac{\log (1/v)}{i+\log s}+ \frac{\{\log (1/v)\}^2}{\{i+\log s\}\{i+\log (1/v)+\log s\}}\right) \\
% & = \frac{1}{i+\log s}\left(1-\frac{\log (1/v)}{i+\log s}+ \frac{\{\log (1/v)\}^2}{\{i+\log s\}^2}\right. \\
% & \quad \left. -
% \frac{\{\log (1/v)\}^3}{\{i+\log s\}^2\{i+\log (1/v)+\log s\}}\right)
%\end{align*}
%where
%\begin{align*}
%\left| \frac{\{\log (1/v)\}^3}{\{i+\log s\}^2\{i+\log (1/v)+\log s\}}\right|
% \leq \frac{|\log (1/v)|^3}{\{i+\log s\}^2}
%\end{align*}
%for $i\geq 1$.

\begin{lemma}\label{Hiineq.2}
For $s>1$ and $v\leq s$,
\begin{align*}
& \frac{i+\log s}{i+\log s+\log(1/v)} \geq
  1+\frac{\log v}{i+\log s}- \frac{|\log v|^3}{(i+\log s)^2} \\
& \frac{i+\log s}{i+\log s+\log(1/v)}  \leq
  1+\frac{\log v}{i+\log s}+ \frac{|\log v|^2+|\log v|^3}{(i+\log s)^2},
%& \frac{1}{i+\log s+\log(1/v)} \\
%& \geq
%  \frac{1}{i+\log s}\left(1+\frac{\log v}{i+\log s}- \frac{|\log v|^3}{(i+\log s)^2}\right) \\
%& \frac{1}{i+\log s+\log(1/v)}  \\
%& \leq
%  \frac{1}{i+\log s}\left(1+\frac{\log v}{i+\log s}+ \frac{|\log v|^2+|\log v|^3}{(i+\log s)^2}\right),
\end{align*}
and
 \begin{align*}
 \left(\frac{i+\log s}{i+\log s+\log(1/v)}\right)^2 
\leq
  1+2\frac{\log v}{i+\log s}+ 4\frac{\sum_{i=2}^6|\log v|^i}{(i+\log s)^2}.
\end{align*}
\end{lemma}

\begin{lemma}\label{lem:log.poly}
\begin{enumerate}
 \item \label{lem:log.poly.1}
 For $x\in(0,1)$ and any positive $\epsilon$,
\begin{align*}
  |\log x| \leq \frac{1}{\epsilon x^\epsilon}.
\end{align*}
\item \label{lem:log.poly.2} For any positive $k$, 
\begin{align*}
 |\log x|^k %\log \max(x,1/x)
 \leq
 %\frac{1}{\epsilon}\left\{\max(x,1/x)\right\}^\epsilon
\frac{x^{k\epsilon}+ x^{-k\epsilon}}{\epsilon^k},
\end{align*}
 for all $x\in(0,\infty)$ and any positive $ \epsilon$.
\end{enumerate}
\end{lemma}
\begin{proof}
For $x\in(0,1)$, we have
\begin{align*}
% \log\frac{1}{x}\leq \frac{1}{x}-1 \\
% \log\frac{1}{x^\epsilon}\leq \frac{1}{x^\epsilon}-1 \\
 \log\frac{1}{x^\epsilon}\leq \frac{1}{x^\epsilon}-1 
&\Leftrightarrow\  \epsilon\log \frac{1}{x} \leq \frac{1}{x^\epsilon}-1 \\
&\Rightarrow \ |\log x| \leq \frac{1}{\epsilon x^\epsilon}
\end{align*}
for any positive $ \epsilon$. Then Part \ref{lem:log.poly.1} follows.
 Similarly, for $x\in(1,\infty)$,
\begin{align*}
 \log x^\epsilon\leq x^\epsilon-1 
&\Leftrightarrow\  \epsilon\log x \leq x^\epsilon-1 \\
&\Rightarrow \ |\log x| \leq \frac{x^\epsilon}{\epsilon}.
\end{align*}
Then Part \ref{lem:log.poly.2} follows. 
\end{proof}

%%%%%%%%%%%%%%%%%%%%%%%%%%%%%%%%%%%%%%%%%%%%%%
%% Single Appendix:                         %%
%%%%%%%%%%%%%%%%%%%%%%%%%%%%%%%%%%%%%%%%%%%%%%
%\begin{appendix}
%\section*{???}%% if no title is needed, leave empty \section*{}.
%\end{appendix}
%%%%%%%%%%%%%%%%%%%%%%%%%%%%%%%%%%%%%%%%%%%%%%
%% Multiple Appendixes:                     %%
%%%%%%%%%%%%%%%%%%%%%%%%%%%%%%%%%%%%%%%%%%%%%%
%\begin{appendix}
%\section{???}
%
%\section{???}
%
%\end{appendix}

%%%%%%%%%%%%%%%%%%%%%%%%%%%%%%%%%%%%%%%%%%%%%%
%% Support information (funding), if any,   %%
%% should be provided in the                %%
%% Acknowledgements section.                %%
%%%%%%%%%%%%%%%%%%%%%%%%%%%%%%%%%%%%%%%%%%%%%%
 \section*{Acknowledgements}
% The authors would like to thank ...
% 
The first author was supported by partially supported by KAKENHI \#16K00040, 19K11852.
The second author was supported in part by a grant from the Simons Foundation (\#418098 to William Strawderman).
 
%%%%%%%%%%%%%%%%%%%%%%%%%%%%%%%%%%%%%%%%%%%%%%
%% Supplementary Material, if any, should   %%
%% be provided in {supplement} environment  %%
%% with title inside \textbf{} and short    %%
%% description below.                       %%
%%%%%%%%%%%%%%%%%%%%%%%%%%%%%%%%%%%%%%%%%%%%%%
%\begin{supplement}
%\textbf{???}.
%???.
%\end{supplement}

%%  MR numbers will be added by VTeX.                      %%
%%                                                         %%
%%  Use \cite{...} to cite references in text.             %%
%%                                                         %%
%%%%%%%%%%%%%%%%%%%%%%%%%%%%%%%%%%%%%%%%%%%%%%%%%%%%%%%%%%%%%

%% if your bibliography is in bibtex format, uncomment commands:
%\bibliographystyle{imsart-number} % Style BST file (imsart-number.bst or imsart-nameyear.bst)
%\bibliography{bibliography}       % Bibliography file (usually '*.bib')

%% or include bibliography directly:
% \begin{thebibliography}{}
% \bibitem{b1}
% \end{thebibliography}

\end{document}